\documentclass[12pt,journal,onecolumn ]{IEEEtran}

\usepackage{enumerate}

\usepackage{hyperref}
\usepackage[utf8]{inputenc}
\usepackage[margin=1in]{geometry}
\usepackage{graphicx}
\usepackage{amsmath,amssymb,amsthm,authblk,mathrsfs,subfigure}
\usepackage[acronym]{glossaries}
\usepackage{color}
\usepackage{tikz} 
\usepackage[english]{babel}
\usepackage{comment}
\usepackage[nottoc]{tocbibind} 
\usepackage[outdir=./]{epstopdf}
\usepackage{graphicx}
\usepackage{verbatim}
\usepackage{mathtools}
\usepackage{amsthm}
\usepackage{lipsum}

\usepackage{url}
\usepackage{thmtools}
\usepackage{dsfont}
\usepackage{cancel}

\usepackage{setspace}
\usepackage{todonotes}

\newcommand{\be}{\begin{equation}}
\newcommand{\ee}{\end{equation}}

\newcommand {\R}{\mathbb R}

\newcommand {\C}{\mathbb C}

\newcommand{\TK}{\textcolor{magenta}}
\newcommand{\LG}{\textcolor{green!50!black}}

\theoremstyle{plain}

\makeatletter
\newcommand{\sq}[1]{\mathbin{\mathpalette\make@circled{#1}}} 
\newcommand{\make@circled}[2]{%
	\ooalign{$\m@th#1\smallbigcirc{#1}$\cr\hidewidth$\m@th#1#2$\hidewidth\cr}%
}
\newcommand{\smallbigcirc}[1]{%
	\vcenter{\hbox{\scalebox{1.2}{$\m@th#1\square$}}}%
}
\makeatother

\declaretheorem[name={Example},qed={\lower-0.3ex\hbox{$\square$}} ] {Example}

\newcommand{\st}{\, | \,}

\doublespace 

\title{On the gain of entrainment in a class of weakly contractive bilinear control systems
with applications to the   master~equation 
and the ribosome flow model  
}
 \author{Rami Katz and Thomas Kriecherbauer and Lars Gr\"{u}ne 
 and Michael Margaliot\thanks{RK and MM are with the School of Elec. Eng. - Systems, Tel Aviv University, 69978, Israel.    TK and LG  are  with the
Mathematical Institute, University of Bayreuth, Germany.
This research   is partially supported by    research grants from the~DFG (GR 1569/24-1 and KR 1673/7-1,  project number 470999742) and the~ISF. Correspondence: michaelm@tauex.tau.ac.il }}
 
\newtheorem{theorem}{Theorem}

\newtheorem{definition}{Definition}
\newtheorem{corollary}[theorem]{Corollary}

\newtheorem{proposition}[theorem]{Proposition} 
\newtheorem{remark}{Remark}

\newacronym{rfm}{RFM}{Ribosome Flow Model}
\newacronym{pmp}{PMP}{Pontryagin Maximum Principle}
\newacronym{tas}{TASEP}{Totally Asymmetric Simple Exclusion Process}

\newcommand*\diff{\mathop{}\!\mathrm{d}}
\begin{document}
	\maketitle 
	
\begin{abstract}
    We consider a class of bilinear weakly contractive systems 
     that entrain to periodic excitations. Entrainment is important in many natural and artificial processes. For example, in order to function properly
     synchronous generators must entrain to the frequency of the electrical grid, and biological organisms must entrain to the 24h solar day.
    A dynamical system has a positive  gain of 
    entrainment~(GOE) if entrainment also yields a larger output, on average. This property is important in many applications from the periodic 
    operation of bioreactors to the periodic production of proteins during the cell cycle division process. We derive a closed-form formula for the GOE to first-order in the control perturbation. This is used to show that in the class of systems that we consider the GOE is always a 
    higher-order phenomenon. 
    We demonstrate the theoretical results  using  two applications: the  master equation and a model from systems biology called the ribosome flow model, both with time-varying and periodic transition rates. 
\end{abstract}
	\newpage
\begin{IEEEkeywords}
Contractive systems,   
mRNA translation, totally asymmetric simple exclusion process~(TASEP), 
entrainment to periodic excitations, master equation, Markov chains,  Poincar\'e map.
\end{IEEEkeywords}

\section{Introduction}

Nonlinear contractive  systems~\cite{sontag_cotraction_tutorial,LOHMILLER1998683} share many properties with asymptotically stable linear systems.  For example, if the vector field of a contractive system is time-varying and~$T$-periodic then the system admits a unique $T$-periodic solution~$\gamma$ that is exponentially globally asymptotically stable~(EGAS) \cite{entrain_trans}. In particular, 
if the vector field of the contractive system 
is time-invariant    
 then the system admits a unique equilibrium  that is~EGAS.
 In the case when $T$-periodicity of the vector field is driven   by a~$T$-periodic excitation, the convergence to the~$T$-periodic   EGAS solution~$\gamma$ implies that the system \emph{entrains} to the excitation. This property is important in many scientific
 fields
 including: (1)~internal processes in biological organisms  that entrain to the 24h solar day~\cite{entrain_arabid};
 (2)~seasonal outbreaks 
 of epidemics due to periodicity in contact rates~\cite{soper}; 
 (3)~synchronous generators that entrain to the frequency of the electric
 grid; and (4)~brain wave synchronization
 between interacting
 people~\cite{sync_minds2023}. 

 \subsection{Sensitivity of entrainment }
A natural and important question is: what is the sensitivity of~$\gamma$ to small perturbations  in the periodic control?
In other words, if~$u$ and~$u+\Delta u$
are two $T$-periodic controls, differing by a small  (in 
some appropriate sense) perturbation~$\Delta u$,
what can be said about the difference  between  
the corresponding $T$-periodic  
solutions~$\gamma^u$ and~$\gamma^{u+\Delta u}$? Such questions  are important, for example, when designing  synthetic biology 
oscillators    (see, e.g.,
\cite{Hasty2002}). However, addressing these questions rigorously is   a non-trivial task, as typically there is no explicit description of the periodic solutions, that is,
the mapping~$u\mapsto \gamma^u$ is not   explicitly known. Moreover, the perturbation~$\Delta u$ belongs to an infinite-dimensional vector space. Thus, any analysis of the difference between
$\gamma^{u+\Delta u}$ and~$\gamma^u$, for a ``small''~$\Delta u$,
requires using an infinite-dimensional operator
mapping~$\Delta u$ to the ``difference''~$\gamma^{u+\Delta u}-\gamma^{u}$.

Pavlov et al.~\cite{pavlov} considered contractive systems 
(more precisely, the \LG{closely related and slightly more general} class of convergent systems~\cite{RUFFER2013277}) that are excited by  the  output~$\omega(t,\omega_0)$ of a linear harmonic oscillator (where~$\omega_0 $ denotes the initial condition) and showed that there exists a \emph{continuous} mapping~$\alpha$ such that
\[
\gamma^\omega(t)=\alpha(\omega(t,\omega_0)).
\]
Pavlov et al.~\cite{pavlov} refer to this property as the \emph{frequency response}
of the contractive system. This result implies that~$\gamma^u$ is continuous in~$u$. However, in this paper we want to study the difference between $\gamma^u$ and~$\gamma^{u+\Delta u}$ for small $\Delta u$. For this task we do not only need continuity, but also differentiability of $\gamma^u$ with respect to $u$, in a suitable functional analytic sense. The analysis in this paper will provide such a differentiability result and moreover closed-form formulas for the corresponding derivatives.


\subsection{Gain of entrainment }
A closely related topic  is   the \emph{gain of entrainment}~(GOE) problem, that is,   the question of
whether   applying periodic controls may lead to a ``better'' output, on average, than equivalent constant controls.
To explain this, consider the nonlinear   system 
\begin{align*}
\dot x&= f(x,u) ,  \\
y&=h(u,x),
\end{align*}
with state~$x\in\R^n$, input~$u\in\R^m$, and   a scalar output~$y\in\R$.  Here the output represents a quantity that should be maximized.
For example, in the ribosome flow model~(RFM) described in Section~\ref{sec:app2_the_RFM}, $y(t)$ is the protein production rate at time~$t$. 

Assume that for any~$T>0$ and any admissible $T$-periodic control~$v$ the system admits a corresponding $T$-periodic EGAS solution~$\gamma^v$. This property holds in particular  for contractive systems (see Section~\ref{sec:pre} below). 
 Note that this property  implies that for constant controls, $\gamma$ is just an equilibrium point.  
The   average output along the $T$-periodic solution is then 
\[
\bar y(v) :=\frac{1}{T}\int_0^T y(\gamma^v(t),v(t))\diff t.
\]
Let~$\bar v:=\frac{1}{T}\int_0^T v(t)\diff t$. Then for the constant control~$u(t)\equiv \bar v$ the system admits   an
EGAS equilibrium~$e^{\bar v}$, so the corresponding average   output is~$h(e^{\bar v},\bar v )$.  We say that the system admits a~GOE for the  periodic control~$v$ if
\[
         \bar y(v) >  h(e^{\bar v},\bar v) 
\]
In other words, we   consider 
two controls~$u^1(t)=v(t)$ and~$u^2(t)\equiv \bar v$ that have  the same average value, and compare the resulting average outputs,   along the corresponding~EGAS trajectories.     GOE implies that in terms of maximizing the output, 
 the periodic control~$u^1$ is preferable to the ``equivalent'' constant control~$u^2(t)\equiv \overline{v}$ (see Fig.~\ref{fig:goe}).

\begin{figure}
    \centering
    \includegraphics[width=\linewidth]{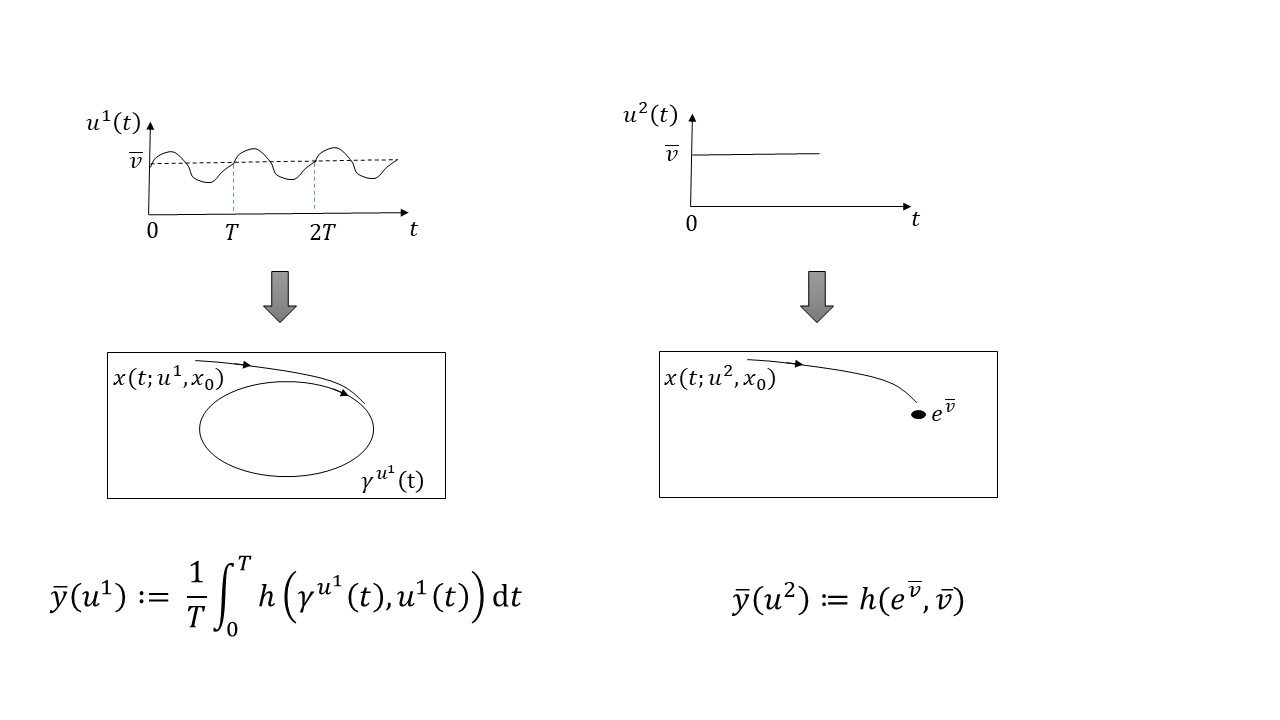}
    \caption{Gain of entrainment (GOE) in a system with output~$y=h(x,u)$. Left: $u^1(t)$ is a $T$-periodic control and~$\gamma^{u^1}(t)$, $t\in [0, T)$, is the corresponding EGAS trajectory. 
    Right: $u^2(t)  \equiv \overline v:=\frac{1}{T}\int_0^T u^1(t)\diff t$ is a constant control and~$e^{\overline v}$ is the corresponding EGAS equilibrium. The system admits a GOE if~$\bar y(u^1)>\bar y(u^2)$.  }
    \label{fig:goe}
\end{figure}

GOE may be  relevant   in numerous applications including   vehicular  control using     
periodic traffic lights, periodic fishery~\cite{peri_fishing}, periodic operation of chemical processes~\cite{Periodic_Chemical_Reactors}, periodic outbreaks of epidemics~\cite{seasonal_epidemics}, periodic gene expression~\cite{entrain_trans}, and more.  All these processes can also be controlled using constant controls, and assessing the GOE can be used to determine if periodic controls are ``better'' than constnt ones.

Asymptotically stable linear time-invariant~(LTI)
systems admit no~GOE.
The next example demonstrates this.
\begin{Example}\label{exa:no_GOE_linear}
Consider the single-input single-output system:
\begin{align*}
\dot x(t)&=Ax(t)+bu(t),\\
y(t)&=c^T x(t),
\end{align*}
where~$A$ is a Hurwitz matrix. 
Let~$G:\C\to\C$ denote the 
transfer function of this system, i.e.,
$G(s):=c^T( s I-A)^{-1}b$. Fix a frequency
$\omega >0$, and consider the 
system response  to two controls
\[
u(t)\equiv \bar v, 
\]
and
\[
u(t)+\Delta u(t)= \bar  v+\sin(\omega  t).
\]
Let~$T:=2\pi / \omega$. For the first   control, the output converges to the $T$-periodic solution
\[
\gamma ^u(t) \equiv G(0) \bar  v,
\]
and for the second control to
\be\label{eq:seconcont}
\gamma ^{u+\Delta u}(t) = G(0) \bar v  + |G(\mathrm{i}\omega) |\sin \bigl (   \omega t  + \arg G(\mathrm{i}\omega ) \bigl ).
\ee
The average values of these outputs are  
$    \frac{1}{T}\int_0^T     \gamma^u(t)
    \diff t =G(0)\bar v,
$ 
and
$    \frac{1}{T}\int_0^T     \gamma^{u+\Delta u}(t)
    \diff t=G(0)\bar v,
$ 
implying no GOE. 
\end{Example}

Non-linear systems may   admit a GOE. The next example from~\cite{pavlov} demonstrates this.
\begin{Example}\label{exa:nonlin_with_goe}
    Consider the non-linear system
    \begin{align}\label{eq:nolonpav}
    \dot x_1&=-x_1+x_2^2,\nonumber\\
    \dot x_2&=-x_2+u,\nonumber\\
    y&=x_1.
    \end{align}
This is the series interconnection of two (scalar)  contractive  systems (see Section~\ref{sec:pre}),  and is thus a  contractive system~\cite{ofir2021sufficient}.
Consider two controls. The first 
is~$u(t)\equiv 0$. For this control the output of~\eqref{eq:nolonpav} converges to the steady-state value~$\gamma^u(t)\equiv 0$.
The second control is~$u +\Delta u$, with~$\Delta u(t): =a\sin(\omega t)$, 
where~$  a,\omega>0$. This control  is~$T$-periodic for~$T:=2\pi/\omega$, and has average~$\frac{1}{T}\int_0^T (u(t)+\Delta u(t)) \diff t = 0$.
A calculation shows that for this control 
the output of~\eqref{eq:nolonpav} converges to a peridoic solution~$\gamma^{u+\Delta u}$ satisfying 
\be\label{eq:yss_exa}
2(4\omega^4+5\omega^2+1) 
\gamma^{u+\Delta u}(t)= 4a^2\omega^2+a^2-2a^2\omega\sin(2t\omega-2\tan^{-1}(\omega))-a^2\cos( 2t\omega-2\tan^{-1}(\omega)),  
\ee
so 
\be\label{eq:quad_goe}
\frac{1}{T}\int_0^T \gamma^{u+\Delta u}
(t) \diff t 
 = \frac{a^2 }{2(1+\omega^2)} .
\ee
Thus, the GOE for this input is always positive and depends on both  the excitation amplitude~$a$ and frequency~$\omega$. Note that the GOE is quadratic in~$a$.
\end{Example}

Examples~\ref{exa:no_GOE_linear}
and~\ref{exa:nonlin_with_goe} 
raise several interesting  questions. For example, given a nonlinear system, how can one determine if it admits a~GOE and for what controls? What control gives the best possible~GOE?
Are there systems that have (or do not have) a GOE for any admissible control? 

Several recent papers analyzed the GOE in a nonlinear
model from systems biology called the ribosome flow model~(RFM). This 
  is a non-linear  ODE model for the flow of ribosomes
  and the production of 
proteins during mRNA translation~\cite{GOE_RFM_2023,max_period_output_RFM}, where the output is the protein production rate. It was shown that, perhaps surprisingly, for several special cases the RFM has no~GOE. In other words, $T$-periodic  controls are not better than constant ones in terms of maximizing the average protein production rate.

GOE may also be studied in the framework of optimal
control theory. The basic idea behind this approach is to pose the problem of finding an admissible control, within the set of $T$-periodic control inputs, which maximizes the average output under: (1)~a constraint on the average of the control,  and (2)~considering the output along the EGAS $T$-periodic solution. However, these constraints make the problem difficult to tackle, and solutions exist only for   special  cases,   e.g.,    for scalar systems~\cite{rapo_convex,rapoport_sing}.

\subsection{Contributions of this paper}
Motivated by the considerations in the previous subsection, in this paper
we rigorously analyze 
 the sensitivity of entrainment and the~GOE in a specific class of bilinear and weakly contractive systems. 
The main contributions of this paper include:
\begin{itemize}
    \item   We analyze the sensitivity function of solutions with respect to a change in the control input. The latter is given as the solution of an ODE on an appropriate infinite-dimensional Banach space. Rather than analyzing  this infinite-dimensional ODE directly, we reduce the analysis to an ODE in $\mathbb{R}^n$ and employ tools from finite-dimensional systems, thereby making the analysis more accessible to a wider audience;
    \item We then show that   the mapping $u\to \gamma^u (0) $ from a periodic control~$u$ to the 
    initial value of the corresponding
    $T$-periodic solution~$\gamma^u$   admits a  Fr{\'e}chet     derivative, and that this derivative is continuous. We derive an explicit expression for the Fr{\'e}chet     derivative. This can be used in conjunction with numerical tools to study various 
    properties of the mapping~$u\to  \gamma^u$. In particular, it provides new formulas that may reduce the numerical work for finding controls that improve the output of the system;
  \item We show that   at constant controls
  this Fr{\'e}chet  derivative is zero, implying that 
  their~GOE is inherently a 
  higher-order phenomenon in the norm  of the control perturbation. We consider this a surprising phenomenon, which is due to the interplay of the first-order terms of the periodic perturbation $\Delta u$ entering the system and the resulting change of the initial condition $\gamma^{u+\Delta u}(0)$ of the corresponding periodic solution, which effectively cancel out in the integrated output; 
 \item    We demonstrate these
results by applying them to two important mathematical models: 
 the master equation for irreducible finite state Markov processes
 and the RFM, 
both with~$T$-periodic transition rates. It should be noted that the master equations obey a weak version of contractiveness. Indeed, the equations are only non-expanding in general, and irreducibility of the underlying Markov-process is required to guarantee  some form of contraction. Therefore our proofs avoid making unnecessary use of contraction. In fact, we found that entrainment for all admissible controls is the basic assumption,
and this must be  supplemented by a non-degeneracy condition for the associated  Poincar\'e map.
\end{itemize}

The remainder of this paper is organized as follows.
The next section   introduces some  notation and several known results that will be used later on. Section~\ref{sec:the_model}
describes the bilinear control system that we analyze. 
Section~\ref{section:main_results} details our main results.  Section~\ref{sec:proofs} includes the proofs of the main results. Sections~\ref{sec:app2_the_RFM} and~\ref{sec:appl1_ME} 
detail  applications of the theoretical results to two models:
the~RFM and time-varying  
Markov chains.
Section~\ref{sec:generalizations}
describes a generalization of our results. 
The final section concludes, and describes   possible topics for further research. 

\section{Preliminaries}\label{sec:pre}
We begin with a short review on contraction in finite-dimensional ODEs. For more details and proofs, see, e.g.,~\cite{sontag_cotraction_tutorial,Sontag2010,bullo_contractive_systems}. 

\subsection{Contraction in finite-dimensional  ODE systems}\label{sec:contraction}
Consider 
the nonlinear dynamical system:
\begin{equation}\label{eq:gene_sys}
\begin{array}{lll}
&\dot{x}(t) =f(x(t),u(t)) ,  \\
&x(0) = x_0\in  \Omega,
\end{array}
\end{equation}
with state~$x(t)\in \Omega \subseteq \R^n$, and
control input~$u(t)=\begin{bmatrix}u_1(t)&\dots&u_m(t)\end{bmatrix}^T\in U\subseteq \mathbb{R}^m$.
  Let~$\mathcal{U}\subseteq Z$ denote the subset of admissible controls, where $Z$ is a Banach space.  We assume that the vector field~$f\in C^1$,  and let
  \begin{align*}
      f_x(x,u):=\frac{\partial}{\partial x} f(x,u) ,\quad f_u(x,u):=\frac{\partial}{\partial u} f(x,u),  
  \end{align*}
denote its Jacobians  with respect to the state and the control, respectively. 

 For~$u\in \mathcal{U}$ and $x_0\in \Omega$, let~$x(t;u,x_0)$ denote the 
 solution of~\eqref{eq:gene_sys} at time $t\geq 0$. We say that the system~\eqref{eq:gene_sys} is   contractive if  there exist~$\eta>0$   (called the contraction rate) and a norm~$|\cdot|:\R^n\to [0,\infty)$ such that for any 
 admissible control~$u$, and  any two initial conditions~$a,b\in\Omega$, we have 
\begin{equation}\label{eq:ContractionProperty}
|x(t;u,a)-x(t;u,b)|\leq \exp (-\eta t)|a-b|, \text{ for all } t\geq 0 .
\end{equation}
In other words, any two trajectories approach each other at an exponential rate,  and this holds uniformly with  respect to the controls. 

Recall that the norm~$|\cdot|:\R^n\to [0,\infty) $ induces a matrix norm~$\|\cdot\|:\R^{n\times n}\to [0,\infty)$ defined by~$
\|A\|:=\max_{|x|=1}|Ax|,
$
and a matrix measure $\mu:\R^{n\times n}\to \R$ (also called logarithmic norm~\cite{strom1975logarithmic}) defined by
\begin{equation}\label{matrix_measure}
 \mu(A):=\lim_{\varepsilon\downarrow 0}\frac{\| I_n+\varepsilon A \| -1}
{\varepsilon}.   
\end{equation}
Assume that the state space~$\Omega$ is forward-invariant, convex and bounded, and that~$U\subseteq \mathbb{R}^m$ is bounded. Then a sufficient condition for  contractivity of~\eqref{eq:gene_sys}
is that   the  matrix measure~$\mu:\R^{n\times n} \to \R$ induced by~$|\cdot|$ satisfies 
\begin{equation}\label{eq:mu1}
     \mu  \left(f_x(x,u)\right) \leq -\eta <0\text{ for all } x\in\Omega, u\in U.
\end{equation}
In both applications described in this paper, namely,  the RFM and time-varying Markov chains, neither~\eqref{eq:mu1} nor \eqref{eq:ContractionProperty} hold uniformly for all of the forward invariant state space. Therefore we do not use these bounds and their powerful consequences like the input-to-state stability (ISS) property~\cite{desoer,Sontag2010}. Instead, we use entrainment to periodic controls
as our main hypothesis. This property    has already been established for the application systems in this paper, see \cite{RFM_entrain,entrainME}.

\subsection{Sensitivity functions, abstract ODEs and evolution equations}\label{subsec:Sensitivity}
 Recall that $x(t;u,x_0)$ is the solution of~\eqref{eq:gene_sys} at time~$t\geq 0$.
Consider the sensitivity function~$\frac{\partial x}{\partial x_0}(t):=\frac{\partial x}{\partial x_0}(t;u,x_0)$.
Formally, $\frac{\partial x }{\partial x_0}(t)$ belongs  to the set of bounded linear 
operators from~$\R^n$ to~$\R^n$  
(and  can thus be considered as an~$n\times n$ matrix),  and describes how a small perturbation in the initial condition $x_0$ translates to a change in the solution value $x(t;u,x_0)$ at time~$t$. It is well-known that $\frac{\partial x}{\partial x_0}:[0,\infty)\rightarrow \mathbb{R}^{n\times n}$ satisfies the variational (matrix) ODE:
\begin{align*}
\frac{d}{dt} \frac{\partial x(t)}{\partial x_0}& = f_x (x(t),u(t))  \frac{\partial x(t )}{\partial x_0},  \\
\frac{\partial x(t)}{\partial x_0}(0)&=I_n.
\end{align*}

Given a Banach space of controls~$Z$ and a set of admissible controls $\mathcal{U}\subseteq Z$, the   sensitivity with respect to controls  at time~$t$ is
denoted by~$\frac{\partial x(t)}{\partial u}$. Then
$
\frac{\partial x(t)}{\partial u}  \in \mathcal{B}(Z,\R^n)$, 
where~$\mathcal{B}(Z,\R^n)$ is the space of bounded linear operators mapping $Z$  to $\mathbb{R}^n$. The operator~$\frac{\partial x(t)}{\partial u}$ describes how a perturbation in the control input translates to a change in the solution value~$x(t;u,x_0)$. Differently from~$\frac{\partial x(t)}{\partial x_0}$, the domain of $\frac{\partial x(t)}{\partial u}$   lies in an \emph{infinite-dimensional} space. Assuming for the moment that~$\frac{\partial x}{\partial u}:[0,\infty)\rightarrow \mathcal{B}(Z,\mathbb{R}^n)$ 
exists and is differentiable (in an appropriate sense) for all~$t\geq 0$, we expect it to satisfy two properties. First,~$\frac{\partial x}{\partial u}(0)=0$, as~$x(0;u,x_0)=x_0$ does not depend on the control. Second,   a formal calculation gives:
\begin{align}\label{eq:SensitivityU}
\frac{d}{dt} \frac{\partial x(t )}{\partial u} =\frac{\partial }{\partial u}  f( x(t ),u(t))= f_x (x(t),u(t))  \frac{\partial x(t )}{\partial u} + 
f_u (x(t),u(t)) .
\end{align}
These considerations suggest the 
abstract ODE (evolution equation) over the Banach space $X=\mathcal{B}(Z,\mathbb{R}^n)$:
\begin{align}\label{eq:abst_c}
\dot \psi(t)&=A(t)\psi(t)+\lambda(t),\nonumber \\
\psi(0)&=0,
\end{align}
where $A(t):= f_x (x(t),u(t))$
are $n\times n$ matrices and $\lambda(t) := f_u (x(t),u(t))$.    In \eqref{eq:abst_c} we interpret~$A(t)$ as a map $\mathcal{B}(Z,\mathbb{R}^n)\ni \psi \mapsto A(t)\psi \in \mathcal{B}(Z,\mathbb{R}^n)$; that is for all $z\in Z$, $\left[A(t)\psi \right](z) = A(t)\psi(z)$. If~$f(x,u)$ satisfies the infinitesimal condition for contraction~\eqref{eq:mu1}
then, by definition, so does~$A(t)$ in~\eqref{eq:abst_c}.

The study of evolution equations of the form \eqref{eq:abst_c} relies on the theory 
of $C_0$ semigroups of bounded linear operators and their generators, see, e.g., \cite{pazy_semigroups}. Yet, 
in this paper we prove existence and uniqueness of a classical solution to \eqref{eq:SensitivityU}, for a class of  weakly 
contractive bilinear control systems, using an alternative approach  that is more tractable and, in particular, involves only finite-dimensional~ODEs.  Indeed, for $\Delta u\in Z$, set $z(t):=\frac{\partial x(t)}{\partial u}\Delta u\in \mathbb{R}^n$. Then $z(t)$ is the solution of the finite-dimensional system $\dot{z}(t)=A(t)z(t)+\lambda(t)\Delta u(t)$ with zero initial condition. However, this is only true if existence of $\frac{\partial x(t)}{\partial u}$ has been established. Below, we will consider the system for $z(t)$ as a starting point and prove that $B(t):\Delta u\mapsto z(t)$ is a bounded linear map, which is a candidate for $\frac{\partial x(t)}{\partial u}$. We then show that 
\begin{equation*}
x(t;u+\Delta u,x_0) = x(t;u ,x_0)+B(t)\Delta u + o(\left\|\Delta u \right\|_Z),
\end{equation*}
which demonstrates that indeed $B(t) =\frac{\partial x(t)}{\partial u}$ is the Fr{\'e}chet derivative. This proof strategy relies only on tools from finite-dimensional systems.
For contraction analysis of
infinite-dimensional systems, see~\cite{SRINIVASAN2023203} 
and the references therein.

\subsection{An implicit function theorem on Banach spaces}
The following implicit function theorem will be central to our analysis (see e.g. \cite[Section VI.2]{edwards2012advanced} and \cite[Section 4]{rheinboldt1969local}):
\begin{theorem}\label{Thm:IFT} [Implicit function theorem on Banach space.]
Let $M$, $N$, and~$G$ be Banach spaces.
Let~$D\subseteq M\times N$ be an open set, and
let~$f:D \to G$ be a continuously Fr{\'e}chet differentiable mapping. 
Assume that 
for some~$(a,b)\in D$, we have~$f(a,b)=0$ and that
the~Fr{\'e}chet derivative~$d_nf(a,b):N\to G$ is a bijection. 
Then there exists a neighborhood $W_a \subseteq M$ of $a$,
a neighborhood $W_b\subseteq N$ of $b$
and a continuously Fr{\'e}chet differentiable mapping $g:W_a\to W_b$ such that
\[
f(x,y)=0 \iff y=g(x),  \text{ for all }  (x,y)\in W_a \times W_b.
\]
\end{theorem}

\section{The model} \label{sec:the_model}
In this section, we
describe 
the 
bilinear control system that we consider, and our standing assumptions.  The control system is of the form
%
\begin{align}\label{eq:SystParam}
\dot{x}(t)& =    f(u(t))+M(x(t)) u(t)  ,   \\
 y(t)&= h(  u(t),x(t)) ,\nonumber
\end{align}
and we denote by~$x(t;u,x_0)$ the solution of~\eqref{eq:SystParam} that satisfies in addition the initial condition~$x(0)=x_0$.
Note that~\eqref{eq:SystParam} is not the  standard affine in the control system, as the ``drift vector'' $f$ depends on~$u$ and not on~$x$. However, as shown in Sections~\ref{sec:app2_the_RFM} and~\ref{sec:appl1_ME}, several important real-world models can be written  in the form~\eqref{eq:SystParam}.
Also, an  extension   of our results to more general models is described in
Section~\ref{sec:generalizations}.

We begin by defining the ambient space~$Z$ for the set of admissible controls~$u(t)
\in \mathbb{R}^m$. 
Given $T>0$, let $Z:=C_{\text{Per}}(0,T)$ denote the space of continuous functions $g:[0,T] \rightarrow \mathbb{R}^m$ satisfying~$g_i(0)=g_i(T)$ for all $i=1,\dots, m$. Then any~$g\in Z$ can be extended to a continuous function on $[0,\infty)$ via periodicity, and it is this extension that is used in the control system \eqref{eq:SystParam}. The set~$Z$, equipped with the norm 
 \begin{equation*}
 \left\|g \right\|_Z : = \max_{i=1,\dots,m} \; \max_{t \in [0,T]} |g_i (t)|,  \end{equation*}
is a Banach space. 
We are now ready to formulate our assumptions on the functions appearing in the control system~\eqref{eq:SystParam} with state space in~$\R^n$ and scalar output~$y$. 
\subsection{General Assumptions}\label{assumptions}
 There exist nonempty, open sets~$\Omega \subset \R^n$ and~$\mathcal{U} \subset Z$
such that conditions~{\bf (C1)} -- {\bf (C3)} hold.
\begin{itemize}
\item[\bf{(C1)}] {\em Regularity.} The maps~$f:\mathbb{R}^m\rightarrow \mathbb{R}^n$, $M: \Omega \to \R^{n\times m}$, and~$h:\mathbb{R}^m\times \Omega \rightarrow \mathbb{R}$ are continuously differentiable.
\item[\bf{(C2)}] {\em Entrainment.} For every~$u\in \mathcal{U}$ 
system~\eqref{eq:SystParam} has a unique~$T$-periodic solution~$\gamma^{u}$ in~$\Omega$. 
\item[\bf{(C3)}] {\em Non-degeneracy.} For every~$u\in \mathcal{U}$ the sensitivity function~$\frac{\partial x} {\partial x_0}(T;u,\gamma^{u}(0))$ at time~$T$ for the periodic solution~$\gamma^u$ does not have~$1$ as an eigenvalue.
\end{itemize}
\subsection{Discussion of the Assumptions and first conclusions}
 We now describe a setting in which our general assumptions are always satisfied.
Suppose we have a  contractive system in the sense that the bound~\eqref{eq:ContractionProperty}
holds 
for all~$a$,~$b$ in a compact and forward invariant subset~$K \subset \mathbb{R}^n$ and
for all admissible controls~$u \in \mathcal{U}$. Let us further assume that the boundary of~$K$ is repelling in the sense that solutions cannot remain in~$\partial K$ for all times~$t\geq 0$ and that a solution that lies in the interior of~$K$ at some time~$t^*$ will stay there for all times~$t > t^*$. Then the interior~$\Omega := \operatorname{Int}(K)$ of~$K$ is also forward  invariant and condition~{\bf (C2)} is satisfied. Indeed, as the Poincar\'e maps~$P^u:K \to K$ defined by~$P^u(x_0) : =  x(T;u,x_0)$ is a contraction  on $K$ there exists for every admissible control~$u$ a unique fixed point~${x}^u\in K$ of~$P^u$. The solution starting in~${x}^u$ at time~$0$ provides the unique periodic solution~$\gamma^u$. The assumption of a repelling boundary as detailed above implies that all of~$\gamma^u$ lies in~$\Omega$. Condition~{\bf (C3)} can be verified  by contradiction. If there exists~$u \in \mathcal{U}$ and an eigenvector~$v\in \mathbb{R}^n \setminus \{0\}$ with~$\frac{\partial x} {\partial x_0}(T;u,\gamma^{u}(0))v=v$, then inequality~\eqref{eq:ContractionProperty} is violated for~$t=T$, $a=\gamma^{u}(0)$, and $b=\gamma^{u}(0)+\varepsilon v$ for~$\varepsilon$ sufficiently small.

 The conditions that we have posed as general assumptions are weaker than the one described in Subsection \ref{sec:contraction}. Our choice of conditions was motivated by the applications described in Sections~\ref{sec:app2_the_RFM} and~\ref{sec:appl1_ME} below. A crucial ingredient for the proofs of our main results, which follows from our general assumptions, is stated in the following proposition.
\begin{proposition}\label{PROP:Replacing_ISS}
Assume conditions~{\bf (C1)} -- {\bf (C3)} hold.  Fix $u_0 \in \mathcal{U}$. 
Then there exist numbers~$d$, $r_x$, $r_{u} >0$ such that for all~$(u, x_0)$ in the open set~$\mathcal{D}_{u_0} := B(u_0, r_u)\times B(\gamma^{u_0}(0), r_x) \subset \Omega \times \mathcal{U}$
the solution~$x(t;u,x_0)$ of system~\eqref{eq:SystParam} with initial condition $x(0)=x_0$ stays within a compact set~$K_{u_0} \subset \Omega$ for all times $t\in [0, T]$. Moreover, for all~$(u, x_0)$, $(\tilde{u}, \tilde{x}_0) \in \mathcal{D}_{u_0}$ and $t\in [0, T]$ we have
    \begin{equation}\label{eq:ISS_replacement2}
|x(t;u,x_0)-x(t;\tilde{u},\tilde{x}_0)| \leq d \left( |x_0-\tilde{x}_0| +  \|u-\tilde{u}\|_Z \right).
\end{equation}
Finally, for constant controls~$u_0(t)\equiv \bar v$ with $u_0 \in \mathcal{U}$, we have that~$\gamma^{u_0}(t)\equiv e^{\bar v}$, that is,  an equilibrium point.
\end{proposition}
\begin{proof}
Choose~$\varepsilon > 0$ such that the closure~$K_{u_0}$ of the~$\varepsilon$-neighborhood~$\Omega_\varepsilon := \cup_{t\in [0,T]} B(\gamma^{u_0}(t),\varepsilon)$ of the compact set~$\gamma^{u_0}([0,T])$ is contained in~$\Omega$. 
For~$x_0 \in B(\gamma^{u_0}(0),\varepsilon)$ and~$u\in B(u_0, 1)$ we consider the difference~$\Delta x(t) := x(t;u,x_0)-\gamma^{u_0}(t)$ for~$t \in [0,T^*)$ with~$T^*=T^*(u, x_0)$ chosen so that~$[0,T^*)$ is the maximal interval for which~$|\Delta x(t)| < \varepsilon$ holds. Let 
\begin{align}\label{eq:Jxu}
   J(x,u) : = \frac{\partial M(x)} {\partial x}u  
\end{align}
denote the Jacobian of the vector field  in~\eqref{eq:SystParam}
with respect to~$x$. Then $\Delta x$ satisfies the linear initial value problem
\begin{align*}
 \frac{d}{dt} \Delta x(t) &= A(t) \Delta x(t) + b(t) ,\\
 \Delta x(0) &= x_0-\gamma^{u_0}(0)
 \end{align*}
 with 
\begin{align*} 
  A(t)&: = \int_0^1 J(\gamma^{u_0}(t) + r \Delta x(t), u_0(t)) \diff r  , \\
  b(t)&: = f(u(t))-f(u_0(t)) + M(x(t;u,x_0)) \bigl(u(t)-u_0(t)\bigl), 
\end{align*}
where we used the fact that~$A(t)\Delta x(t)=\bigl (  M(x(t;u;x_0)) - M(\gamma^{u_0}(t)) \bigl) u_0(t)$ by the fundamental theorem of calculus, and the convexity of the balls~$B(\gamma^{u_0}(t),\varepsilon)$. Condition~{\bf (C1)} together with the compactness of~$K_{u_0}$ implies that there exists a constant~$d' = d'(u_0) >0$ such that $|b(t)| \leq d' \|u-u_0\|_Z$ for all~$u \in B(u_0, 1)$. 
Moreover, for all~$r\in [0,1]$, $t \in [0,T^*)$ the argument~$(\gamma^{u_0}(t) + r \Delta x(t), u_0(t))$ of the continuous map~$J$ in the definition of~$A(t)$ is contained in some compact set~$K_{u_0}\times \hat K_{u_0} \subset \Omega \times \mathbb{R}^m$ that only depends on~$u_0$.
Thus there exists a constant~$\hat d = \hat d(u_0) > 0$ such that the transfer matrix~$\Phi (t;s)$ of the linear time-varying system~$\dot{z}=A(t)z$ satisfies~$\|\Phi(t;s)\| \leq \hat d$ for~$0\leq s \leq t < T^*$. By the variation of constants formula we have
\begin{align*}
    \Delta x (t) = \Phi(t;0) \left ( x_0-\gamma^{u_0}(0) \right ) + \int_0^t \Phi(t;s) b(s) \diff s
\end{align*}
and therefore
\begin{align}\label{eq:auxiliaryISS}
    | \Delta x (t) | \leq \hat d \, |x_0-\gamma^{u_0}(0)| \,+ \, t \hat d d' \,\|u-u_0\|_Z, \text { for all }t\in[0, T^*). 
\end{align}
Choosing~$r_u := \min\{ \varepsilon/(2 T \hat d d'), 1 \} $ and~$r_x:=\min\{ 1/\hat d, 1\}  \varepsilon/2$ ensures that the solution~$x(t;u,x_0)$ cannot leave the neighborhood~$\Omega_\varepsilon$ for~$t \in [0, T]$. In other words~$T^* > T$.

  In order to prove the bound~\eqref{eq:ISS_replacement2} we proceed as above, with the only difference that we replace~$u_0$ by~$\tilde{u}$ and, consequently,~$\gamma^{u_0}(0)$ by~$\tilde{x}_0$ and~$\gamma^{u_0}(t)$ by~$x(t;\tilde{u},\tilde{x}_0)$. Observe that the convexity of the balls~$B(\gamma^{u_0}(t),\varepsilon)$ implies that the interpolations~$r x(t;u,x_0) + (1-r) x(t;\tilde{u},\tilde{x}_0)$, $0\leq r\leq 1$, that appear in the analogue formula for matrix~$A(t)$ above are still contained in the set~$\Omega_\varepsilon \subset \Omega$. One may again derive estimate~\eqref{eq:auxiliaryISS} with constants $d'$, $\hat d$ only depending on $u_0$. Choosing $d:=\hat d \max \{ T d', 1 \} $ yields claim~\eqref{eq:ISS_replacement2}. 

For a constant control~$u_0 \equiv \bar v$ system~\eqref{eq:SystParam} is autonomous and~$Y(t):=\frac{d}{dt} \gamma^{u_0}(t)$ solves the initial value problem
 \begin{align}
\dot Y(t)&= J\bigl (\gamma^{u_0}(t),\bar v \bigl) Y(t), \nonumber\\
Y(0) &= \frac{d}{dt} \gamma^{u_0}(0) . \nonumber
\end{align} 
Observe that the fundamental matrix~$\Phi(t)$ for this system, with~$\Phi(0)=I_n$, equals the sensitivity function~$\frac{\partial x} {\partial x_0}(t;u_0,\bar v)$. Using periodicity of~$\gamma^{u_0}$  gives~$Y(0)=Y(T)=\Phi(T) Y(0)$. Condition~{\bf (C3)} implies that~$Y(0)=0$ because~$1$ is not an eigenvalue of~$\Phi(T)$. Thus~$Y(t)=\Phi(t) Y(0)=0$ for all~$t\in [0, T]$. This shows that~$\gamma^{u_0}(t)$ is a constant function.
\end{proof}

\subsection{The map $\Gamma : u \mapsto \gamma^{u}(0)$ and a Poincar\'e-type map}
According to condition~{\bf (C2)} we can uniquely define $\Gamma:\mathcal{U}\rightarrow \Omega$  by 
\begin{equation}\label{eq:GammaDef}
 \Gamma(u) := \gamma^{u}(0),
\end{equation}
that
maps a~$T$-periodic control
to the initial condition at time zero of the corresponding~$T$-periodic solution. 
It is crucial for our analysis to prove differentiability of the map~$\Gamma$ and to find an expression for its derivative. The key observation to achieve this is that for every~$u \in \mathcal{U}$, we have~$\Gamma(u)=\gamma^u(T)=x(T;u,\Gamma(u))$. This relation characterizes~$\Gamma(u)$,  as we have assumed the uniqueness of the periodic solution for any admissible control. Thus~$\Gamma(u)$ is the unique zero of the Poincar\'e-type map $P:\mathcal{D} \to \R^n$ defined by
\begin{equation}\label{eq:defPoinc}
  P( u,x_0):= x(T;u,x_0)-x_0,
\end{equation}
with $\mathcal{D}:=\bigcup_{u \in \mathcal{U}} \mathcal{D}_u$, 
where the open sets~$\mathcal{D}_u \subset \mathcal{U} \times \Omega$ are defined as in Proposition~\ref{PROP:Replacing_ISS}.

 Let us revisit previous examples where we can read off~$\Gamma(u)$ for some controls $u$ from the explicit solution formulas.
In Example~\ref{exa:no_GOE_linear} above, we find that for the controls~$u,u+\Delta u$ defined there, we have
\begin{equation*}
\Gamma(u) = G(0)\bar v,\qquad
 \Gamma(u+\Delta u) = G(0)\bar v+ \operatorname{Im}(G(\mathrm{i}\omega )).     
\end{equation*}
  Similarly, in Example~\ref{exa:nonlin_with_goe}, we obtain
  \begin{equation*}
   \Gamma(u) = 0,\qquad
 \Gamma(u+\Delta u) =  \frac{a^2\omega ^2 
(2\omega^2+5)}{(\omega^2+1)^2
(4 \omega^2+1)}. 
  \end{equation*}

We conclude this section with a very simple example that satisfies conditions~{\bf (C1)} -- {\bf (C3)}.

\begin{Example}\label{exa:simple_scalar_syst}
    Consider the scalar  equation with state space~$\Omega=\mathbb{R}$
    \begin{align}\label{eq:scalar_sys}
\dot x= - u x,
    \end{align} 
    where~$u(t)$ takes values in~$(c,\infty)$, $c>0$, for all~$t\geq0$.
This is in the form~\eqref{eq:SystParam} with~$n=m=1$,~$f=0$,
and~$M(x)=-x$.
The solution of~\eqref{eq:scalar_sys} is given by
\be\label{eq:expsolsca}
x(t; u , x_0)=\exp(-w(t) ) x_0 , 
\ee
with~$w(t):=\int_0^t u(s)\diff s$. As~$w(T) > 0$, condition~{\bf (C3)} holds. Moreover,
\[
P( u,x_0):= x(T;u,x_0)-x_0= \bigl (\exp(-w(T) ) -1\bigl )x_0 .
\]
yields that the Poincar\'e-type map~$P(u,\cdot)$ has a unique zero~$\Gamma (u)=0$, so 
condition~{\bf (C2)} holds. Finally, in this case we 
can compute the derivatives of~$P$ explicitly,
\begin{equation*}
    \frac{\partial P}{\partial x_{0}} (u, x_0) =\exp(-w(T) ) -1\,,\qquad
     \frac{\partial P}{\partial u} (u, x_0) \Delta u = \exp(-w(T))   x_0 \left (-\int_0^T\Delta u(s) \diff s \right )
\end{equation*}
for~$\Delta u \in Z$ $($recall that~$\frac{\partial P}{\partial u} (u, x_0)$ lies in~$\mathcal{B}(Z,\R))$. The second formula can be derived from 
\begin{align*}
 x(T;u+\Delta u, x_0) -x(T; u , x_0)&=  \left(  \exp(-\int_0^T\Delta u(s) \diff s) -1 \right )  \exp(-w(T))   x_0 \\
 &=\left (-\int_0^T\Delta u(s) \diff s \right )\exp(-w(T))   x_0+o(\|\Delta u\|_Z),
\end{align*}
where $o(\epsilon)$ denotes  the set of functions~$g:\R_+\to\R$ satisfying
$
\lim_{\epsilon\to 0^+} \frac{g(\epsilon)}{\epsilon}=0.
$
\end{Example}

\section{Main results}\label{section:main_results}
This section describes the  main results. All the proofs are placed in the next section.
Throughout the present section we assume that the general assumptions of Section~\ref{assumptions} are satisfied. Moreover, we use freely the notation introduced in Proposition~\ref{PROP:Replacing_ISS} and the definitions provided by \eqref{eq:GammaDef} and \eqref{eq:defPoinc}.

Modifying a control~$u$ to~$u+\Delta u$
generates a change from~$\gamma^u$ to~$\gamma^{u+\Delta u}$.
Since~$\Gamma(u)$ and~$\Gamma(u+\Delta u)$ are the initial condition of these periodic solutions, respectively, it is important to know the derivative of~$\Gamma$. The next result provides an explicit expression for this quantity. 

\begin{theorem}[Explicit expression for the derivative of~$\Gamma$]
\label{thm:Differentiability}
Consider the system~\eqref{eq:SystParam}
. Then the mapping $\Gamma $ in~\eqref{eq:GammaDef}
is continuously Fr{\'e}chet differentiable, and  its  Fr{\'e}chet derivative $d\Gamma(u)\in \mathcal{B}(Z,\mathbb{R}^n)$ is given by  
\begin{align}\label{eq:dgamma}
d\Gamma(u)\Delta u &= -\left  ( \frac{\partial x}{\partial x_0} (T;u,\Gamma(u)) - I_n \right  ) ^{-1}    \frac{\partial x}{\partial u} ( T;u,\Gamma(u))\Delta u\nonumber\\
&=-\left  ( \Phi_{u,\Gamma(u)}(T;0) - I_n \right  ) ^{-1}\int_0^T\Phi_{u,\Gamma(u)}(T;s)\Bigl ( M(\gamma^u(s))+Df(u(s)) \Bigl ) \Delta u(s)\diff s
\end{align}
for all $\Delta u\in Z$. 
Here $\Phi_{u,x_0}(t;t_0)$ is the  
solution at time~$t$ 
of the linear time-varying matrix ODE:
\begin{align}\label{eq:FundEq}
\dot{\Phi}(\tau) &= J\bigl (x(\tau;u,x_0),u(\tau) \bigl)\Phi(\tau),\nonumber\\
\Phi(t_0) &= I_n
\end{align}
with $J$ from \eqref{eq:Jxu}.
\end{theorem}

\begin{Example}
Consider again the system in Example~\ref{exa:simple_scalar_syst}.
For this system, Eq.~\eqref{eq:dgamma} gives
\begin{align*} 
d\Gamma(u)\Delta u &= -\left  ( \frac{\partial x}{\partial x_0} (T;u,\Gamma(u)) - I_n \right  ) ^{-1}    \frac{\partial x}{\partial u} ( T;u,\Gamma(u))\Delta u\nonumber\\
& = -\left  (\exp(-w(T)) -1\right )^{-1} 
\bigl  (-\int_0^T \Delta u(s) \diff s \bigl  )\exp(-w(T))   \Gamma(u)\\
&=0,
\end{align*}
and this is   correct, as in this case~$\Gamma(u)\equiv 0$.
\end{Example}

\begin{remark}
Pavlov et al.~\cite{pavlov}  considered the more general class of convergent systems and showed that for such systems the map~$u\mapsto \gamma^u$ is \emph{continuous}. Theorem~\ref{thm:Differentiability}
shows that for the class of systems we study, and assuming that conditions~{\bf (C1)--(C3)} hold, $\Gamma$ satisfies a stronger regularity condition. Furthermore, as we will see below the explicit expression for~$d\Gamma$  plays a crucial role in the 
analysis of the~GOE. We note that 
Pavlov et al.~\cite{pavlov} 
did not assume condition {\bf (C3)}, but it is possible to 
demonstrate that this condition is necessary to guarantee that Theorem~\ref{thm:Differentiability} holds.
\end{remark}

  The next result analyzes~$d\Gamma (u) $ in the particular case  where~$u$ is a  constant control.
\begin{corollary}\label{coro:fixed_con1}
 Let~$u(t)\equiv \bar  v$ be a constant control in~$\mathcal{U}$.
 Denote by~$e^{\bar v}$ the corresponding equilibrium point
   of the system according to
Proposition~\ref{PROP:Replacing_ISS} so, in particular, 
 $\Gamma(u)=e^{\bar v}$. 
Let~$H:=J(e^{\bar v}, \bar v)$.  Then  
\be\label{eq:dg_const_cont1}
d\Gamma(u)\Delta u =- \left  ( \exp(HT )  - I_n \right  ) ^{-1}    \int_0^T \exp \left ( H(T-s)  \right  ) \left(M(e^{\bar v}) +{  Df(\bar v)}\right)\Delta u(s)\diff s 
\ee
for any~$\Delta u\in Z$. 
\end{corollary}

The next  result    considers the difference in the outputs along the 
 periodic solutions corresponding  to two controls:~$u$ and~$u+\Delta u$.
\begin{theorem}[A first-order expression for the GOE]\label{thm:GOEGeneral}
Consider the system \eqref{eq:SystParam}. Fix~$u\in \mathcal{U}$.  For all~$u+\Delta u$ with~$(u+\Delta u, \Gamma(u+\Delta u)) \in \mathcal{D}_{u}$
let~$\gamma^u$ and~$\gamma^{u+\Delta u}$ denote the corresponding 
$T$-periodic solutions, and consider 
the   outputs
\begin{align}\label{eq:avgoutput}
y^u(t)&:=h\left(u(t),\gamma^u(t)\right),\nonumber \\
y^{u+\Delta u}(t)&:= h\left(u(t)+\Delta u(t),\gamma^{u+\Delta u}(t)\right).
\end{align}
Let
\[
\overline{y^{u+\Delta u} - y^u } : = 
\frac{1}{T}
\int_0^T \left (  y^{u+\Delta u}(t)- y^u(t) \right ) \diff t
\]
denote  the difference in the average
of these outputs. Then 
\begin{align}\label{eq:FirstOrdTay}
\overline{y^{u+\Delta u} - y^u }  & = \frac{1}{T}\int_0^T \frac{\partial h}{\partial u}(u(t),\gamma^u(t))\Delta u (t) \diff t\nonumber\\
&+\frac{1}{T}\int_0^T\frac{\partial h}{\partial x}(u(t),\gamma^u(t))\left( 
\frac{\partial x}{\partial x_0}(t;u,\Gamma(u))d\Gamma(u)\Delta u+ \frac{\partial x}{\partial u}(t;u,\Gamma(u))\Delta u  \right ) \diff t\nonumber\\&+o(\left\|\Delta u  \right\|_Z).  
\end{align}
\end{theorem}
 \begin{remark}
Given $u\in \mathcal{U}$, it is often of interest to numerically verify whether a $T$-periodic perturbation $\Delta u$ in the control yields  an increase in  the average output. A straightforward approach is to use a numerical ODE solver to converge to  the
solutions~$\gamma^u$ and~$\gamma^{u+\Delta u}$ and then numerically compute the average of the difference
$y^{\gamma+\Delta \gamma}-y^{\gamma }$. However, if the system dimension and/or  the set of candidate perturbations are large, this may be a costly procedure.

Combining Theorem~\ref{thm:GOEGeneral} 
with 
the  explicit expressions derived 
for~$\frac{\partial x}{\partial x_0}(t;u,\Gamma(u))$, $\frac{\partial x}{\partial u}(t;u,\Gamma(u))$ and $d\Gamma(u)$   yields, via Fubini's theorem, the following kernel representation
\begin{align}\label{eq:KernelRep}
\overline{y^{u+\Delta u} - y^u }  & = \frac{1}{T}\int_0^T \mathcal{K}(t)\Delta u(t) \diff t +o(\left\|\Delta u  \right\|_Z),
\end{align}
where
\begin{align*}
\mathcal{K}(t) &:= \frac{\partial h}{\partial u}(u(t),\gamma^u(t))- \left(\int_0^T \frac{\partial h}{\partial x}(u(s),\gamma^u(s)\Phi_{u,\Gamma(u)}(s;0))\diff s\right)\left( \Phi_{u,\Gamma(u)}(T;0) - I_n \right  ) ^{-1}\nonumber \\
&\times \Phi_{u,\Gamma(u)}(T;t)\Bigl(M(\gamma^u(t))+Df(u(t)) \Bigl)+\left(\int_t^T\frac{\partial h}{\partial x}(u(s),\gamma^u(s))\Phi_{u,\Gamma(u)}(s;t) \diff s\right)\nonumber \\
&\times \Bigl(M(\gamma^u(t))+Df(u(t)) \Bigl)\nonumber .
\end{align*}
The latter suggests an approach for maximizing the GOE (to first order). Indeed, for $\left\|\Delta u \right\|_Z\leq 1$, we have
\begin{align*}
\left|\frac{1}{T}\int_0^T \mathcal{K}(t)\Delta u(t) \diff t \right|\leq \frac{1}{T}\int_0^T \left|\mathcal{K}(t) \right|_1\diff t.
\end{align*}
Thus, to increase the first term on the right-hand side of \eqref{eq:KernelRep},  one can take $\Delta u\in Z$ which approximates $\operatorname{sgn}(\mathcal{K}(t))^{\top}=\left(\operatorname{sgn}(\mathcal{K}_1(t)),\dots, \operatorname{sgn}(\mathcal{K}_m(t)) \right)^{\top}$ on $[0,T]$. Note that the numerical approximation of $\mathcal{K}$ requires computing only $\gamma^u$ and $\Phi_{u,\Gamma(u)}$.

\end{remark}

 The next result uses  Corollary~\ref{coro:fixed_con1}
to analyze the~GOE in the vicinity of a constant control.
\begin{theorem}
[GOE w.r.t.  a constant control
is a  higher-order  phenomenon]
\label{thm:dv}
Consider the system~\eqref{eq:SystParam}. Fix a constant control~$u(t)\equiv \bar v$, and let~$e^{\bar v}$ denote the corresponding equilibrium point.
 For all~$u+\Delta u$ with~$(u+\Delta u, \Gamma(u+\Delta u))$, 
 let~$\gamma^{u+\Delta u}$ denote the corresponding $T$-periodic solution.  Let
\begin{equation}\label{eq:avgoutputConst}
\begin{array}{lll}
&y^u(t):=h(\bar{v},e^{\bar v}), \\
&y^{u+\Delta u}(t):= h\left(u(t)+\Delta u(t),\gamma^{u+\Delta u}(t)\right), \quad t\in [0,T),
\end{array}
\end{equation}
that is, the outputs along the  constant and
$T$-periodic solution, respectively. 
Then for any~$\Delta u$, such that~$\overline{\Delta u}:= \frac{1}{T}\int_0^T \Delta u(t)  \diff t=0$, we have 
\begin{equation}\label{eq:GOE1stZero}
\frac{1}{T}
\int_0^T \left (  y^{u+\Delta u}(t)- y^u(t) \right ) \diff t =o(\|\Delta  u \|_Z).  
\end{equation}
\end{theorem}

To explain why this result is important, assume for a moment that
\begin{equation*}
 \frac{1}{T}
\int_0^T \left (  y^{u+\Delta u}(t)- y^u(t) \right ) \diff t =r (\Delta u)+ o(\|\Delta u\|_Z),   
\end{equation*}
where~$r:Z \to \R$ is \emph{linear} in~$\Delta u$, and~$r(\Delta u)\not= 0$. This would imply that      either the perturbation~$\Delta u$
or~$-\Delta u$   always yields  a~GOE. 
Theorem~\ref{thm:dv} states that the linear term~$r(\Delta  u)$ is  zero. 
Thus, the problem of finding a control that yields a GOE, if it exists, is non trivial. 
Furthermore, Theorem~\ref{thm:dv} 
implies that the GOE for constant controls is determined by terms that are at least second-order in~$\Delta u$, suggesting that the study of GOE may be related to some form of ``convexity'' of higher-order operators (see~\cite{katriel2019optimality} for some related considerations).

\section{Proofs}\label{sec:proofs}
We begin with the  proof of Theorem~\ref{thm:Differentiability}. This requires several auxiliary results.
\subsection{Analysis of the partial derivatives~$\frac{\partial x}{\partial u}(t;u,x_0)$ and~$\frac{\partial x}{\partial x_0}(t;u,x_0)$}\label{Sec:partialu}
The partial derivative~$\frac{\partial x}{\partial u}(t;u,x_0)$ is a bounded linear operator and an element of~$\mathcal{B}(Z,\mathbb{R}^n)$. It satisfies an abstract ODE in this space. However, as explained in Section~\ref{subsec:Sensitivity} we begin by considering the time evolution of the directional derivatives~$\frac{\partial x}{\partial u}(t;u,x_0) \phi \in \mathbb{R}^n$ for~$\phi \in Z$ and discuss the question of Fr{\'e}chet differentiability later. The presented analysis is inspired by the discussion on 
``linearizations compute differentials''  in
\cite[Section 2.8]{sontag2013mathematical}.

More precisely, fix~$u_0 \in \mathcal{U}$. We seek to determine~$\frac{\partial x}{\partial u}(t;u,x_0)\phi$ for~$t \in [0, T]$ and~$(u,x_0)$ in the neighborhood~$\mathcal{D}_{u_0}$ of~$(u_0, \gamma^{u_0}(0))$ specified in Proposition~\ref{PROP:Replacing_ISS}. The direction~$\phi$ is represented by~$\Delta u \in Z$ where~$(u+\Delta u, x_0)$ is also required to lie in~$\mathcal{D}_{u_0}$. We consider the difference of the two corresponding solutions
\[
\Delta x(t) := x(t;u+\Delta u,x_0)-x(t;u,x_0).
\]
To simplify the notation, let
$
x(t):=x(t;u,x_0)$, $\tilde u(t):=u(t)+\Delta u(t)$, 
$\tilde x(t):=x(t;\tilde u,x_0)$. Proceeding as in the derivation of~\eqref{eq:ISS_replacement2} in Proposition~\ref{PROP:Replacing_ISS} we obtain 
\begin{align*}
 \frac{d}{dt} \Delta x (t) &= A(t) \Delta x(t) + b(t),\\ \Delta x(0) &= 0,
\end{align*}
 with
 \begin{align*}
   A(t) &= \int_0^1 J(x(t) + r \Delta x(t), u(t)) dr, \\
   b(t) &= f(\tilde u(t))-f(u(t)) + M(\tilde x(t)) \Delta u (t).
\end{align*}
Moreover, Proposition~\ref{PROP:Replacing_ISS} provides a constant~$d>0$ only depending on the choice of $u_0$ such that    
\begin{equation}\label{eq:ISS_specialized}
|\Delta x(t)| \leq  d \|\Delta u\|_Z ,
\text{ for all } t \in [0, T]. 
\end{equation}
It is our task to find, in an appropriate sense, the best approximation $z$ for~$\Delta x$ that depends linearly on~$\Delta u$ for the given pair~$(u, x_0) \in \mathcal{D}_{u_0}$. Rewrite the above initial value problem in the form
\begin{align*}
 \frac{d}{dt} \Delta x (t)&= A_0(t) \Delta x (t)+ b_0(t) + R(t),\\
 \Delta x(0) &= 0,
 \end{align*}
 with 
 \begin{align*}
   A_0(t) & := J(x(t), u(t)), \\
  b_0(t) &:= Df(u(t))\Delta u(t) + M(x(t)) \Delta u (t), \\
  R(t) &:= b(t) - b_0(t) + (A(t)-A_0(t)) \Delta x(t).
\end{align*}
Observe that~$A_0$ does not depend on~$\Delta u$ and that~$b_0$ is a linear function of~$\Delta u$. As we see below~$R$ collects terms that are of higher order in~$\Delta u$. This suggests that the best approximation~$z$ should satisfy the initial value problem
\begin{align*}
 \dot{z}& = A_0(t) z + b_0(t),\\
 z(0) &= 0,
\end{align*}
and this is the same equation that is derived by an abstract approach at the end of Section \ref{subsec:Sensitivity}, although with different notation. Recalling the definition of the transfer matrix~$\Phi_{u,x_0}(t;t_0)$ 
in~\eqref{eq:FundEq} 
the solution $z$ of the initial value problem is given by
\begin{equation}\label{eq:SolutionContSens}
z(t) = \int_0^t\Phi_{u,x_0}(t;s)\Bigl ( M(x(s;u,x_0))+Df(u(s)) \Bigl ) \Delta u(s)\diff s
\end{equation}
and depends linearly on $\Delta u$. Finally, we bound the deviation~$e:=\Delta x - z$. We use that $e$ solves the initial value problem
\begin{align*}
 \dot{e} (t)&= A_0(t) e (t)+ R(t),\\
 e(0) &= 0,
\end{align*}
together with the representation
\begin{align*}
    R(t)&= \left[\int_0^1 Df(u(t) + r\Delta u(t)) - Df(u(t))\diff r\right] \Delta u(t)\\
    &+ \left[\int_0^1 J(x(t) + r \Delta x(t), \Delta u(t)) \diff r\right] \Delta u (t)\\
    &+ \left[\int_0^1 J(x(t) + r \Delta x(t), u(t)) - J(x(t),u(t)) \diff r\right] \Delta x(t).
\end{align*}
For~$0\leq r \leq 1$ we have by construction that~$x(t) + r \Delta x(t)$ is contained in the compact set~$K_{u_0}$ obtained in Proposition~\ref{PROP:Replacing_ISS} and~$u(t) + r \Delta u(t)$ lies in some compact subset of~$\mathbb{R}^m$. Condition~{\bf (C1)} thus implies uniform continuity of~$Df$ and~$J$ and the boundedness of $J$. Together with inequality~\eqref{eq:ISS_specialized} and the linearity of $J$ in the second argument we conclude~$R(t)=o(\|\Delta u\|_Z)$ uniformly for all~$t \in [0,T]$. Since~$e(t) = \int_0^t\Phi_{u,x_0}(t;s) R(s) \diff s$ we 
learn that~$e(t)=o(\|\Delta u\|_Z)$ uniformly for all~$t \in [0,T]$. We summarize our findings.
\begin{proposition}\label{prop:psi_eq}
Fix~$u_0 \in \mathcal{U}$ and denote by~$\mathcal{D}_{u_0}$ the open set in $\mathcal{U}\times\Omega$ specified in Proposition~\ref{PROP:Replacing_ISS}. For any~$(u,x_0) \in \mathcal{D}_{u_0}$ and~$\Delta u\in Z$ with $(u+\Delta u,x_0) \in \mathcal{D}_{u_0}$, we have 
\begin{equation}\label{eq:FrechState}
 x(t;u+\Delta u,x_0) = x(t;u,x_0)+z(t) + o(\left\|\Delta u \right\|_Z), \text{ uniformly over } t\in[0,T],    
\end{equation}
where~$z(t)$, given in~\eqref{eq:SolutionContSens}, is linear in~$\Delta u$. In particular, $x(t;u,x_0)$ is Fr{\'e}chet differentiable with respect to $u$ and $z(t) = \frac{\partial x}{\partial u}(t;u,x_0) \Delta u$. Moreover, $z$ solves the initial value problem
\begin{align}\label{eq:psi_odenew1}
\dot{z}(t)
&= J(x(t;u,x_0),u(t))  z(t) + \Bigl ( M( x(t;u,x_0) )+Df(u(t))\Bigl )\Delta u(t),\\
z(0) &= 0.\nonumber
\end{align}
\end{proposition} 

The next result describes the continuity of $\frac{\partial x}{\partial u}(t;u.x_0)$ as a function of the initial condition $x_0$ and control $u$. 
\begin{proposition}\label{prop:con1}
 Fix~$u_0 \in \mathcal{U}$ and denote by~$\mathcal{D}_{u_0}$ the open set in $\mathcal{U}\times\Omega$ specified in Proposition~\ref{PROP:Replacing_ISS}.
The mapping~$\mathcal{D}_{u_0} \ni  (u,x_0) \mapsto \frac{\partial x}{\partial u}(t;u,x_0) \in \mathcal{B}(Z,\mathbb{R}^n)$ is continuous, uniformly in~$t\in[0,T]$, i.e., for any~$(u,x_0)\in \mathcal{D}_{u_0}$, we have 
\begin{equation*}
\lim_{\substack{
\tilde u \to u \\ 
\tilde x_0\to x_0}} \left(\sup_{t\in [0,T]}\left\|\frac{\partial x}{\partial u}(t;\tilde  u ,\tilde x_0)-\frac{\partial x}{\partial u}(t;u,x_0) \right\|_{\mathcal{B}(Z,\mathbb{R}^n)}\right) = 0.
\end{equation*}
\end{proposition}
\begin{proof}

Fix~$(u, x_0) \in \mathcal{D}_{u_0}$.
Let~$\Delta x_0 \in\mathbb{R}^n$ and~$\Delta u\in Z$ be small enough such that
~$\tilde x_0:= x_0+\Delta x_0$ and~$\tilde u: = u+\Delta u$ satisfy 
$(\tilde u,\tilde x_0) \in \mathcal{D}_{u_0}$. 
To simplify the notation, set
\begin{align*}
 x(t):=x(t;u,x_0),\quad
\tilde x(t):=x(t;\tilde u,\tilde x_0 ).
\end{align*} 
Let $\phi\in Z$ with $\left\| \phi\right\|_Z=1$
. Recalling \eqref{eq:FundEq} and \eqref{eq:SolutionContSens}, we define  for all~$t\in [0,T]$
\begin{align*}
n(t)&:=\frac{\partial x}{\partial u} (t;u,x_0)\phi = \int_0^t\Phi_{u,x_0}(t;s)\Bigl ( M(x(s))+Df(u(s)) \Bigl ) \phi(s)\text{d}s,\\
\tilde n(t)&:=\frac{\partial x}{\partial u} (t;\tilde u, \tilde x_0)\phi=\int_0^t\Phi_{\tilde u,\tilde x_0}(t;s)
\Bigl ( 
M(\tilde x(s))+Df(\tilde u(s)) \Bigl )
\phi(s)\text{d}s.
\end{align*}
Then,
$
n(t)-\tilde{n}(t)=  I_1(t)+ I_2(t),
$
where
\begin{align*}
    I_1(t)&:= \int_0^t\Phi_{u,x_0}(t;s) 
\Bigl ( 
M(x(s))+Df(u(s))-M(\tilde{x}(s))-Df(\tilde{u}(s)) 
\Bigl ) 
\phi(s)\diff s,\\
    I_2(t)&:= \int_0^t
\left ( 
\Phi_{u,x_0}(t;s) -\Phi_{\tilde{u},\tilde{x}_0}(t;s)
\right ) 
\Bigl ( M(\tilde x(s))+Df(\tilde u(s)) \Bigl ) 
\phi(s)\diff s.
\end{align*}
We bound each of these  integrals separately.
 Note that~$x(s),\tilde{x}(s)
$ are contained in the compact set~$K_{u_0}$ of Proposition~\ref{PROP:Replacing_ISS} 
and~$u(s),\tilde{u}(s)
$ lie in some compact subset~$\hat K_{u_0}$ of~$\mathbb{R}^m$ for all~$s\in [0,T]$. As in the proof of Proposition~\ref{PROP:Replacing_ISS} this implies the existence of a constant~$\hat d = \hat d(u_0) >0$ such that for all~$(\tilde u,\tilde x_0) \in \mathcal{D}_{u_0}$ and~$0 \leq s \leq t \leq T$ we have~$\| \Phi_{\tilde{u},\tilde{x}_0}(t;s)\|\leq \hat d$. 

Consider~$I_1$. Condition~{\bf{(C1)}} implies uniform continuity of~$M$ on~$K_{u_0}$ and of~$Df$ on~$\hat K_{u_0}$
Taking into account that~$\left\|\phi \right\|_Z=1$ together with inequality~\eqref{eq:ISS_replacement2} 
we conclude that~$|I_1 (t)| \rightarrow 0$ uniformly for~$t \in [0, T]$ as~$(\Delta u,\Delta x_0)\rightarrow (0,0)$.

To bound~$I_2$, 
it is enough to bound
\[
\Delta  \Phi(t;s):=\Phi_{u,x_0}(t;s)-\Phi_{\tilde u,\tilde x_0}(t;s),\quad  0\leq s\leq t\leq T. 
\]
Note that~$\Delta  \Phi(s;s)=I-I=0$.
By~\eqref{eq:FundEq},
\begin{align*}
    \frac{d}{dt} \Delta \Phi(t;s) &= J(\tilde x(t),\tilde u(t)) \Delta \Phi(t;s)+
  \Bigl   (
  J( x(t),u(t) )-J(\tilde x(t),\tilde u(t) )
  \Bigl ) 
  \Phi_{u,x_0}(t;s).
\end{align*}
This implies the following representation for~$\Delta \Phi(t;s)$:
\begin{align*}
\Delta \Phi(t;s)= \int_s^t \Phi_{\tilde u,\tilde x_0}(t;\tau)
\Bigl   (
J(x(\tau ),u(\tau))- J(\tilde x(\tau),\tilde u(\tau) )
\Bigr )
\Phi_{u,x_0}(\tau;s)
\diff \tau.
\end{align*}
Using the uniform continuity of~$J$ on the compact set~$K_{u_0} \times \hat K_{u_0}$, together with 
the uniform bound~$\hat d$ on the matrix norms of the transfer matrices~$\Phi_{u,x_0}$ and~$\Phi_{\tilde u,\tilde x_0}$  we conclude that also~$|I_2 (t)| \rightarrow 0$ uniformly for~$t \in [0, T]$ as~$(\Delta u,\Delta x_0)\rightarrow (0,0)$.
This completes the proof of Proposition \ref{prop:con1}.
\end{proof}

The next result analyzes~$ {Y}(t):=\frac{\partial x}{\partial x_0}(t;u,x_0)$. Recall that this sensitivity function is the solution at time~$t\geq 0 $
of the linear time-varying ODE:
\begin{align*}
\dot Y(t) &= J(x(t;u,x_0),u(t))Y(t), \\
Y(0) &= I_n,
\end{align*}
that is, $\frac{\partial x}{\partial x_0}(t;u,x_0)=\Phi_{u,x_0}(t;0)$. The continuous dependence of this quantity on~$(u, x_0)\in \mathcal{D}_{u_0}$
uniformly in~$t\in[0,T]$ was shown as part of the proof of Proposition \ref{prop:con1}. We thus have
\begin{proposition}\label{prop:SMOTTH_BETA1}
  The mapping $\mathcal{D}_{u_0} \ni$  $(u,x_0) \mapsto \frac{\partial x}{\partial x_0}(t;u,x_0)$ is continuous, uniformly in~$t\in[0,T]$.
\end{proposition}
\subsection{Proof of Theorem~\ref{thm:Differentiability}  and of Corollary~\ref{coro:fixed_con1}}\label{Sec:PfThm2}
We can now prove  Theorem \ref{thm:Differentiability}.
Recall the   Poincar\'e-type map~$P:\mathcal{D} \to \R^n$  in \eqref{eq:defPoinc}.
 It follows from Propositions~\ref{prop:con1} and~\ref{prop:SMOTTH_BETA1} and from the definition of the domain~$\mathcal{D}$ that~$P$ is continuously differentiable
and its     Fr{\'e}chet derivative at $(u, x_0)$ is given by \begin{equation}\label{eq:addedinp1}
 \begin{bmatrix}
     \Delta u\\\Delta {x}_0  \end{bmatrix}
     \mapsto \begin{bmatrix}
 \frac{\partial x}{\partial u}(T;u, x_0)    & \frac{\partial x}{\partial x_0}(T;u,x_0)-I_n 
\end{bmatrix} \begin{bmatrix}
    \Delta u\\\Delta {x}_0 
\end{bmatrix}.
\end{equation}
Now, recall that~$P(u,x_0)=0$   iff~$x(t;u,x_0)=\gamma^u(t)$, i.e.,  
  the $T$-periodic 
  trajectory corresponding to~$u$, i.e., iff~$x_0=\gamma^u(0)=\Gamma(u)$.

Let $u\in \mathcal{U}$ and $x_0 = \Gamma(u)=\gamma^{u}(0)$. Then, $P(u,x_0)=0$. 
 Condition~{\bf (C3)} says that no eigenvalue of~$\frac{\partial x}{\partial x_0}(T;u,x_0)$ equals~$1$. Therefore
\begin{align*}    
\frac{\partial P}{\partial x_0}(u,\Gamma(u))= \frac{\partial x}{\partial x_0}(T;u,x_0) -I_n,
\end{align*}
is non-singular and the implicit function theorem can be applied to analyze the zero-set of~$P$. Indeed, by
Theorem~\ref{Thm:IFT}, there  exist a neighborhood~$W_{x_0}$ of~$x_0$, a neighborhood~$W_u$ of~$u$, and a continuously Fr{\'e}chet differentiable mapping~$g:W_u \to W_{x_0}$ such that for every~$x_1\in W_{x_0}$ and~$v\in W_u$, we have that 
\begin{equation}\label{eq:pvv1}
P(v,x_1)=0 \ \iff x_1=g(v).
\end{equation}
We conclude by uniqueness that $\Gamma\vert_{W_u}=g$, which shows that $\Gamma:u\mapsto \gamma^u(0)$ is continuously Fr{\'e}chet differentiable.
Now 
differentiating the identity
\[
P(v,\Gamma(v))=0 
\]
with respect to the control~$v$ and using the chain rule and~\eqref{eq:addedinp1} 
gives
\begin{equation*}
  \frac{\partial x}{\partial u} ( T;v,\Gamma(v))+
\left ( \frac{\partial x}{\partial x_0} (T;v,\Gamma(v)) - I_n \right )   d\Gamma(v)    = 0 ,
\end{equation*}
so for any $\Delta v\in Z$  we obtain with the help of Proposition~\ref{prop:psi_eq}
\begin{align}  \label{eq:dgz}
d\Gamma(v)\Delta v &= -\left  ( \frac{\partial x}{\partial x_0} (T;v,\Gamma(v)) - I_n \right  ) ^{-1}    \frac{\partial x}{\partial u} ( T;v,\Gamma(v))\Delta v \nonumber\\
&= -\left  ( \Phi_{v,\Gamma(v)}(T;0) - I_n \right  ) ^{-1}\int_0^T\Phi_{v,\Gamma(v)}(T;s)\Bigl ( M(\gamma^v(s))+Df(v(s)) \Bigl ) \Delta v(s)\text{d}s .
\end{align}
This completes the proof of Theorem~\ref{thm:Differentiability}. 
 In the case of constant controls~$u(t)\equiv \bar v$ we know from Proposition~\ref{PROP:Replacing_ISS} that~$x(t;u,\Gamma(u))\equiv e^{\bar v}$. The definitions of~$H$ in the statement of Corollary~\ref{coro:fixed_con1} and of the transfer matrix in~\eqref{eq:FundEq} imply~$\Phi_{u,\Gamma(u)}(t;s) = \exp ( H(t-s) )$. This proves Corollary~\ref{coro:fixed_con1}.
\subsection{Proof of Theorem~\ref{thm:GOEGeneral}}
 For~$t\in [0,T]$ introduce~$\Delta \gamma (t):= \gamma^{u+\Delta u}(t)-\gamma^u(t)$. Differentiability of~$\Gamma$ yields
\begin{align*}
 \Gamma(u+\Delta u)-\Gamma(u) = d\Gamma(u)\Delta u + o(\left\|\Delta u\right\|_Z )  = O(\left\|\Delta u\right\|_Z )
\end{align*}
The bound~\eqref{eq:ISS_replacement2} of Proposition~\ref{PROP:Replacing_ISS} then implies that~$\Delta \gamma (t) = O(\left\|\Delta u\right\|_Z )$ uniformly in~$t\in [0,T]$. The differentiablity of the output function~$h$   gives
\begin{align}\label{eq:rep_Deltah}
    &h(u(t)+\Delta u(t), \gamma^u(t)+\Delta \gamma(t)) - h(u(t), \gamma^u(t)) \nonumber \\ &\qquad = \frac{\partial h}{\partial u}(u(t),\gamma^u(t))\Delta u(t) + \frac{\partial h}{\partial x}(u(t),\gamma^u(t)) \Delta \gamma (t) + o(\left\|\Delta u\right\|_Z )
\end{align}
uniformly in~$t\in [0,T]$. Finally, since~$\Delta \gamma (t)=x\bigl (t;u+\Delta u,\Gamma(u+\Delta u)\bigr)-x\bigl (t;u,\Gamma(u)\bigr)$ we learn from Propositions~\ref{prop:con1} and~\ref{prop:SMOTTH_BETA1} and from 
the uniform bounds on the transfer matrices that
\begin{align*}
\Delta \gamma(t) = \frac{\partial x}{\partial x_0}(t;u,\Gamma(u))d\Gamma(u)\Delta u+ \frac{\partial x}{\partial u}(t;u,\Gamma(u))\Delta u + o(\left\|\Delta u\right\|_Z )
\end{align*}
uniformly in~$t\in [0,T]$. The claim of Theorem~\ref{thm:GOEGeneral} follows by taking the time average of~\eqref{eq:rep_Deltah}.

\subsection{Proof of Theorem \ref{thm:dv}}
Since~$u(t)\equiv \bar v$ is constant, we have $\gamma^u(t)\equiv e^{\bar v}$. 
Theorem~\ref{thm:GOEGeneral} and the  assumption
that~$\overline{\Delta u}=0$ yield
\begin{align}\label{eq:tobseum}
\frac{1}{T}\int_0^T   \left (  y^{u+\Delta u}(t)- y^u(t) \right ) \diff t &= 
  \frac{\partial h}{\partial x}( \bar v,e^{\bar v})\left ( 
\frac{1}{T}\int_0^T\frac{\partial x}{\partial x_0}(t;u,e^{\bar v}) \diff t
\right )  
d\Gamma(u)\Delta u \nonumber \\&+\frac{\partial h}{\partial x}(\bar v,e^{\bar v})\left ( \frac{1}{T}\int_0^T \frac{\partial x}{\partial u}(t;u,e^{\bar v} )\Delta u \diff t\right) +o(\left\|\Delta u \right\|_Z).  
\end{align}
 We combine the first and second terms on the right-hand side of 
equation~\eqref{eq:tobseum}   in the form
\begin{align*}
&\frac{\partial h}{\partial x}( \bar v,e^{\bar v})  \left ( 
\frac{1}{T}\int_0^T w(t) \diff t \right) 
\end{align*}
with
\begin{align*}
 \qquad \qquad \qquad w(t)&:=  \frac{\partial x}{\partial x_0}(t;u,e^{\bar v}) d\Gamma(u)\Delta u + \frac{\partial x}{\partial u}(t;u,e^{\bar v} )\Delta u.
\end{align*}
Using~$\frac{\partial x}{\partial x_0}(t;u,e^{\bar v})=\Phi_{u,e^{\bar v}}(t;0)$ and Proposition~\ref{prop:psi_eq} we deduce that $w$ solves the linear equation
\begin{align}\label{eq:diffeqforw}
    \dot w=Hw+\left(M(e^{\bar v})+Df(\bar{v})\right)\Delta u(t),
\end{align}
where we denote~$H:=J(e^{\bar v},\bar v)$ as in Corollary~\ref{coro:fixed_con1}.
Observe that we may write~$w(t)=\frac{\partial W}{\partial v}(t; u)\Delta u$ with $W(t;v):=x(t;v, \Gamma(v))$. The~$T$-periodicity of~$W(\cdot\,; v)$ is inherited by $w$ and we have $\int_0^T\dot w(t) \diff t =0$. As it is assumed that $\overline{\Delta u}=0$ we obtain~$H (\int_0^T w(t) \diff t) =0$ by integrating~\eqref{eq:diffeqforw} over one period. The relation~$\frac{\partial x}{\partial x_0}(T;u,e^{\bar v})=\Phi_{u,e^{\bar v}}(T;0)=\exp (HT)$ together with condition~{\bf (C3)} yield that~$0$ cannot be an eigenvalue of~$H$ and that~$H$ is non-singular. This implies~$\int_0^T w(t) \diff t =0$. Thus the sum of the second and third term on the right hand side of equation~\eqref{eq:tobseum} vanishes, too. This completes the proof of Theorem~\ref{thm:dv}.

\begin{remark}\label{remark:const_con}
In Theorem \ref{thm:dv} we assume that the average of the $T$-periodic perturbation satisfies~$\overline{\Delta u}=0$ in order to guarantee  that~$\overline  u= \overline{u+\Delta u}$. However, if we allow perturbations whose average is not necessarily
zero then  the above proof
can be used to determine the perturbation~$\Delta u$ that optimizes, to first order in~$\Delta u$, the average difference between the periodic outputs corresponding to~$u$ and to~$u+\Delta u$. 
Indeed, in this case we learn from equations~\eqref{eq:tobseum} and~\eqref{eq:diffeqforw} that
\begin{align}\label{eq:Texplici}
 \frac{1}{T}\int_0^T \left (  y^{u+\Delta u}(t)- y^u(t) \right ) \diff t &=
 \left (
 \frac{\partial h}{\partial u}( \bar v,e^{\bar v}) - \frac{\partial h}{\partial x}( \bar v,e^{\bar v})
 H^{-1} \left(M(e^{\bar v})+Df(\bar{v})\right )
 \right )
 \overline{\Delta u} \\&+ o(\left\|\Delta u \right\|_Z).\nonumber 
\end{align}
It is interesting to note that,
to first order in~$\Delta u$,
the maximization only depends on the average~$\overline{\Delta u}\in \mathbb{R}^m$, so 
one can always restrict the search to a constant control perturbation  rather than a more general (non-trivial) $T$-periodic control perturbation. This reduces the infinite-dimensional optimization 
problem to a finite-dimensional one. 
 Moreover, 
the optimal~$\overline{\Delta u}$ is just a scaling of the vector 
\[
\left ( \frac{\partial h}{\partial u}( \bar v,e^{\bar v})
 -
 \frac{\partial h}{\partial x}(\bar{v},e^{\bar v})  
  H^{-1} \left(M(e^{\bar v})+Df(\bar{v})\right ) \right)^{\top}.
\]
\end{remark}

\begin{remark}
We assume throughout an~$n$-dimensional state vector and a scalar output~$y(t)=h(u(t),x(t))$.
However, the results can be easily used to study the average difference
\[
T^{-1}\int_0^T ( \gamma^{u+\Delta u}(t)-\gamma^u(t))\diff t.
\]
Indeed, for any~$i\in\{1,\dots,n\}$ setting~$h(u(t),x(t))=x_i(t)$
implies that~$y^u(t)= h(u(t),\gamma^u(t)) = \gamma_i^u(t) $, so for example~\eqref{eq:Texplici} 
gives
\begin{align*}
T^{-1} 
\int_0^T \left (  \gamma_i^{u+\Delta u}(t)- \gamma_i^u(t) \right ) \diff t & = 
 -
 (\alpha^i)^T  
  H^{-1}\left(M(e^{\bar v})+Df(\bar{v})\right)\overline{\Delta u}
  +o(\left\|\Delta u \right\|_Z)  , 
\end{align*}
where~$\alpha^i$ is the~$i$th vector in the standard basis of~$\R^n$.

\end{remark}

\section{Application I: the Ribosome Flow Model}\label{sec:app2_the_RFM}
The ribosome 
flow model~(RFM)
is a phenomenological  
model for the flow of ``particles'' along a one-dimensional ``traffic lane''
that includes~$n$ sites.
The~RFM is the dynamical mean-field approximation of an important model from statistical mechanics called the totally asymmetric simple exclusion process~(TASEP). 
TASEP includes a  1D chain of sites
and particles hop stochastically  
along this chain in a uni-directional manner.
Each site can be either empty or include a particle, and simple exclusion means that a particle cannot hop to a site that is already  occupied. This model has attracted enormous interest, as it is one of the simplest models where phase-transitions appear and can be addressed rigorously, see e.g.,~\cite{solvers_guide,kriecherbauer_krug2010,TASEP_book}. 

The RFM (and its variants)
has been extensively used to model and analyze the flow of ribosomes along the mRNA molecule~\cite{reuveni2011genome,  margaliot2012stability,RFM_SENSE,EYAL_RFMD1,Aditi_abortions,randon_rfm,rfm_feedback,rfm_max,rfmr_2015,RFM_FROM_PDE,alexander2017}, and interconnected RFM networks have been used
to model and study  large-scale mRNA translation in the cell~\cite{Raveh2016, aditi_networks, fierce_compete,HALTER2017267,Ortho_RFM,nani}.  

The RFM is a non-linear model that includes $n$ state-variables. 
The state-variable~$x_i(t)$, $i\in [n]:=\{1,\dots,n\}$, describes  the  density of particles  at site~$i$ at time~$t$. The density is normalized so that~$x_i(t)\in[0,1]$ for all~$t$, and thus
the state space of the RFM is ${K} 
= [0,1]^n$. The flow from site~$i$ to site~$i+1$ at time~$t$
is given by
\[
u_i(t) x_i(t)(1-x_{i+1}(t)),
\]
where~$u_i(t)>0$ is the transition rate from site~$i$ to site~$i=1$, and~$1-x_{i+1}(t)$ is the ``free space'' at site~$i+1$ at time~$t$. In other words, as the occupancy of a given site grows, the flow into this site decreases. This is a ``soft'' version of the simple  exclusion principle in~TASEP. 
The~RFM can be used to 
 model and analyze the evolution of particle
 ``traffic jams'' along the   chain. 

Formally, the RFM is given by the state-equations: 
\begin{equation}\label{eq:ndimrfm}
    \dot x_
    i =u_{i-1}(t) x_{i-1} (
    1-x_{i})-u_{i}(t) x_{i} (
    1-x_{i+1}),\quad i\in [n],
\end{equation}
with~$x_0(t):=1$ 
and~$x_{n+1}(t):=0$.
In other words, the RFM is fed by a reservoir  of particles that is always full, and feeds a reservoir that is always empty.

The output rate from the last site in the~RFM
is
\begin{equation}\label{eq:RFMOutput}
y(t)=h(u(t),x(t)):=u_n(t)x_n(t).    
\end{equation}

When modeling  ribosome  flow, $y(t)$
is the flow of ribosomes exiting the mRNA molecule at time~$t$, 
and thus the protein production rate at time~$t$. Note that we can write the RFM~\eqref{eq:ndimrfm} in the  form~\eqref{eq:SystParam} with $x(t) = \begin{bmatrix} x_1(t)&\dots& x_n(t)\end{bmatrix} ^{T}$, $u(t) =\begin{bmatrix} u_0(t)&\dots& u_n(t)\end{bmatrix}^{T}$, $f(u)\equiv 0$,
and 
\begin{align*}
& M(x)= \begin{bmatrix}
    1-x_1 & -x_1(1-x_2)&0&0&\dots&0&0\\
 0 & x_1(1-x_2)&-x_2(1-x_2)&0&\dots&0&0\\
 \\
&& \vdots\\
 0 & 0&0&0&\dots&x_{n-1}(1-x_n)&-x_n
\end{bmatrix}.
\end{align*}
Let the set $\mathcal{U}$ of admissible controls be given by
\begin{align*}
    \mathcal{U} :  = \left\{u \in Z \st  \  c<u_i(t) \ \text{ for all } t\in [0,T],\ i=1,\dots,m \right\}
\end{align*}
with~$m=n+1$, 
and~$0<c$.
It is known~\cite{RFM_entrain}
that for any~$u\in \mathcal{U}$ the state space~${K}
=[0,1]^n$ is invariant, and its
boundary~{$\partial K$}
is repelling. 
The Jacobian  of the vector field  of the~RFM becomes singular on some points on~{$\partial K$,}
and thus the RFM is not contractive  on~${K}$
w.r.t. any norm. However, for any convex and compact set~$\Omega'\subseteq {\Omega :=}(0,1)^n$, there exists a norm~$\left|\cdot\right|$, that depends on~$\Omega'$,  with associated matrix measure~$\mu$, and a scalar~$\eta=\eta(\Omega')>0$
such that
\begin{equation}\label{eq:ContractionRFM}
\mu\left(J(x,u)\right)\leq-\eta<0 \text{ for all } x\in\Omega',\ u\in [c,C]^{n+1}   
\end{equation}
 for any fixed $C>c$. The latter implies that the RFM is \emph{contractive after a short transient}~\cite{cast_book}, i.e.,
   for any initial condition~$x_0\in [0,1]^n$ and any $T$-periodic control~$u\in \mathcal{U}$ there exists a~$\tau=\tau(u,x_0)\geq 0$ and a compact convex set~$\Omega' = \Omega'(u,x_0) \subseteq (0,1)^n$
such that~$x(t;u,x_0)\in\Omega'$ for all~$t\geq\tau$. In particular, 
for any~$u\in \mathcal U$ the RFM admits a unique~$T$-periodic solution~$\gamma^u$ and~$x(t; u,x_0)$ converges to~$\gamma^u$ 
for any~$x_0\in {K}$. 
There  are biological findings suggesting that  gene expression in the cell  entrains to the periodic cell-cycle program~\cite{peri_cell_cycle,Frenkel2012,Higareda2010,patil2012}.

Note that for every $u\in \mathcal{U}$ the unique~$T$-periodic solution~$\gamma^u$ is contained in some convex and compact set~$\Omega'\subseteq (0,1)^n$, where inequality~\eqref{eq:ContractionRFM} holds. Thus~$\frac{\partial x}{\partial x_0} (T;u,\Gamma(u))=\Phi_{u, \Gamma(u)}(T)$ is a contraction w.r.t.~the corresponding norm and cannot have an eigenvalue equal to~$1$. 

The discussion in  the two previous paragraphs shows that  conditions~{\bf (C1)-(C3)}   in Section~\ref{assumptions} hold for the~RFM. 
Theorem~\ref{thm:dv} now implies
\begin{corollary}
Consider the RFM \eqref{eq:ndimrfm} with output \eqref{eq:RFMOutput}. Fix a constant control~$u(t)\equiv \bar v$, and let~$e^{\bar{ v}}$ denote the corresponding equilibrium point. 
For all~$\Delta u \in Z$ such that~$\overline{\Delta u}=0$ and~$u+\Delta u\in \mathcal{U}$,      let~$\gamma^{u+\Delta u}$ denote the corresponding $T$-periodic solution. Then the average outputs along the $T$-periodic 
solutions $y^u(t) \equiv h(\bar{v},e^{\bar v})$ and $y^{u+\Delta u}(t) = h(u(t)+\Delta u(t),{\gamma}^{u+\Delta u}(t))$ satisfy~\eqref{eq:GOE1stZero}.
\end{corollary}
In other words, to first-order the GOE in the RFM is zero. This is perhaps surprising, as one may expect that by properly coordinating the periodic transition rates along the RFM, it may be possible to increase the average output (even to first-order).

\section{Application II: the master equation }\label{sec:appl1_ME}
In this  section, we demonstrate the theoretical  results derived above using
  an  important mathematical  model, namely, the master equation with time-varying transition rates. Let~$\mathds{1}_m \in\R^m$   denote the vector of all ones.

Consider a system that at each time~$t$
can be in one of~$n$ possible configurations. Let~$z(t)=\begin{bmatrix}
    z_1(t)&\dots &z_n(t)
\end{bmatrix}^T$, where~$z_i(t) \in[0,1]$ 
is the probability that the system is in configuration number~$i$ at time~$t$. Thus,
$\mathds{1}_n^T(t) z(t)=1$ for all~$t\geq 0$. 
The (time-dependent) master equation describes the flow between the possible configurations as the linear time-varying system: 
\begin{align}\label{eq:MasterEqODEs1new}
\dot{z}(t) =  A(u(t))z(t),
\end{align}
with 
\begin{equation}\label{eq:AdefMaster}
A_{i,j}(u(t)) = \begin{cases}
u_{j,i}(t),& j\neq i,\\
-\sum_{k\neq i}u_{i,k}(t),& j=i.
\end{cases}
\end{equation}
Here~$u_{j,i}(t)$, with~$i,j\in[n]$,
is the rate of transition
from configuration~$j$
to configuration~$i$ at time~$t$. Note that the mapping~$u\to A(u)$ is linear. 

For example, for~$n=2$ we have
\be\label{eq:master_eqn2}
\dot z=A(u)z=\begin{bmatrix}
    -u_{1,2}& u_{2,1 }\\
    u_{1,2} & -u_{2,1}
\end{bmatrix}z.
\ee
The first equation here is~$\dot z_1=-u_{1,2}z_1+u_{2,1}z_2$, that is, the change in the probability of being in the first configuration is equal to
 sum of flows from configuration~$1$ to~$2$ ($-u_{1,2}z_1$), 
and from configuration~$2$ to~$1$ ($u_{2,1}z_2$).

The master equation has been used to model and analyze numerous systems and processes in a variety of scientific fields including physics, systems biology, demographics, epidemiology, chemistry, and more, see, e.g.,  the monographs~\cite{haag2017modelling,kampen}. In many of these applications, it is important to consider the case where the   time-varying transitions rates are~$T$-periodic. 

Since $z(t)$ is a probability vector, the state space of~\eqref{eq:MasterEqODEs1new} is the standard $n$-simplex in $\mathbb{R}^n$:
\begin{equation}\label{eq:OmegaStateSpace}
 {K}  
= \left\{y\in [0,1]^n \st \sum_{i=1}^n y_i = 1\right\}.    
\end{equation}

To study the GOE in the $T$-periodic 
master equation, we slightly modify the space of controls. Let $\mathcal{Q}\subseteq [n]^2$ such that $u_{i,j}(t)>0$ for all $(i,j)\in \mathcal{Q}$ and $t\in [0,T]$ implies that $A(u(t))$ is irreducible,  and define $Z$ to be the Banach space satisfying
\begin{enumerate}
    \item $(i,j)\in \mathcal{Q} \Longrightarrow u_{i,j}(0)=u_{i,j}(T),\ u_{i,j}\in C([0,T],\mathbb{R})$.
    \item $(i,j)\notin \mathcal{Q} \Longrightarrow u_{i,j}\equiv 0$
\end{enumerate}
with the norm 
\begin{equation*}
    \left\|u \right\|_Z = \max_{(i,j)\in \mathcal{Q}}\left\|u_{i,j} \right\|_C.
\end{equation*}
We  assume that the set of admissible controls is
\be \label{eq:control_in_ME}
\mathcal{U} :  = \left\{ u \in Z \st    c<u_{i,j}(t)
\ \text{ for all } t\in [0,T],\ (i,j)\in \mathcal{Q} \right\},
\ee
 with~$c>0$.  
 Note that this implies that~$A(u(t)) $ is Metzler and irreducible. 


 We also 
assume  that~\eqref{eq:MasterEqODEs1new} 
admits a scalar output
\[
 y(t) = p(u(t),z(t)),
\]
where
  $p:\mathbb{R}^{n^2}\times \mathbb{R}^n\rightarrow \mathbb{R}$ is continuously differentiable.

Define~$H:\R^n\to \R$ by~$H(y) := \mathds{1}_n^T y$. For $q\in \mathbb{R}$, let
\begin{equation}\label{eq:LevSetH}
    L_H(q) := \left\{w\in \mathbb{R}^n \st H(w) = q \right\},
\end{equation}
that is, the level set of $H$ corresponding to $q$. The function $H$ is a first integral for \eqref{eq:MasterEqODEs1new}, meaning that  for any $q\in \mathbb{R}$, the level set $L_H(q)$ is invariant under the dynamics of \eqref{eq:MasterEqODEs1new}. 

\subsection{Representing the master equation as a bilinear control system}
Eq.~\eqref{eq:MasterEqODEs1new} is not in the form~\eqref{eq:SystParam}.
Furthermore, the state space~{$K$} 
of the master equation satisfies~${K}  
\subset L_H(1)$, and since~$L_H(1)$ is an affine manifold in~$\mathbb{R}^n$, we cannot immediately apply the theoretical results in  Section \ref{section:main_results} to \eqref{eq:MasterEqODEs1new}. 

To represent the master equation in the form~\eqref{eq:SystParam}, we make two modifications. First,   introduce the parallel shift~$\kappa:\R^n\to\R^n$ defined by
\begin{equation}\label{eq:KappaTransf}
 \kappa(w) := w-n^{-1}\mathds{1}_n,
\end{equation}
and the change of variables
\begin{equation*}
x(t) := \kappa(z(t)).
\end{equation*}
In the new coordinates,  the master equation  becomes 
\begin{align}\label{eq:MasterEqODEs2new}
\dot{x}(t) &=  A(u(t))x(t)+n^{-1}A(u(t))\mathds{1}_n,\nonumber \\
y(t)& =h(u(t),x(t))  ,
\end{align}
with~$h(u,x):=p(u,x+n^{-1}\mathds{1}_n)$. 
 Note that~$h $ is continuously differentiable on $\mathbb{R}^{n^2}\times \mathbb{R}^n$.

The state space of~\eqref{eq:MasterEqODEs2new}
is~$\kappa(\TK{K})\subseteq L_H(0)$, 
and~$L_H(0)=(\operatorname{span}(\mathds{1}_n))^{\perp}$ is a linear subspace  that is trivially diffeomorphic to~$\mathbb{R}^{n-1}$.
Second, let
\[
u:=\begin{bmatrix}
    u_{1,1}&u_{1,2}&\dots & u_{1,n}&u_{2,1}&\dots& u_{n,n}
\end{bmatrix}^T,
\]
and define~$f:{\color{violet} \mathbb{R}^{n^2}} \to \R^m$ by
\be\label{eq:defu_master}
f(u):=n^{-1}A(u)\mathds{1}_n .
\ee
Also, let~$\alpha^1,\dots,\alpha^n $ denote the standard basis in~$\R^n$. 
Define a matrix~$M:\R^n\to \R^{n\times n^2}$ as follows. For any~$i\in[n]$, 
the~$i$th row of~$M(x) $ is
\be\label{eq:def_Mx_master}
 \begin{bmatrix}
 M_{i,1}(x)& \dots & M_{i,n^2}(x)
     \end{bmatrix}=
 \begin{bmatrix}  ( c^{i,1}(x) )^T& \dots & (c^{i,n}(x) )^T
\end{bmatrix},
\ee
where every~$c^{i,j}$ is a   vector of length~$n$ defined  by
\[
c^{i,j}(x) := \begin{cases}
x_j \alpha^i, & j\neq i, \\
-x_i\sum_{k\neq i}\alpha^k, &j=i.
\end{cases}
\]
For example, for~$n=2$  we have
\begin{align*}
    M(x)&=\begin{bmatrix}
   ( c^{1,1} (x))^T  & ( c^{1,2}(x))^T \\
   ( c^{2,1} (x))^T& (c^{2,2}(x))^T
\end{bmatrix}\\&=
\begin{bmatrix}
     -x_1 (\alpha^2)^T  & x_2 (\alpha^1)^T \\
    x_1 (\alpha^2)^T & -x_2 (\alpha^1)^T
\end{bmatrix}\\
&=\begin{bmatrix}
    0&  -x_1   & x_2 &0 \\
    0& x_1   & -x_2 &0
\end{bmatrix}.
\end{align*}

It is straightforward to verify that~\eqref{eq:MasterEqODEs2new} can be represented as~\eqref{eq:SystParam} with $f(u)$ defined in~\eqref{eq:defu_master}, and~$M(x)$ defined in \eqref{eq:def_Mx_master}. For example, for~$n=2$, we have
\begin{align*}
    M(x)u &= \begin{bmatrix}
    0&  -x_1   & x_2 &0 \\
    0& x_1   & -x_2 &0
\end{bmatrix}\begin{bmatrix} 
u_{1,1}\\u_{1,2}\\u_{2,1}\\u_{2,2}
\end{bmatrix}\\
&= \begin{bmatrix} 
 -u_{1,2}x_1+u_{2,1} x_2\\
 u_{1,2} x_1-u_{2,1}
x_2\end{bmatrix} 
\end{align*}
(compare with~\eqref{eq:master_eqn2}).

\subsection{Application of Theorem~\ref{thm:dv}}
Once expressed in the form~\eqref{eq:SystParam}, we now show that the general assumptions of Section~\ref{assumptions} hold for the master equation. We begin with the definition of the state space
\begin{align*}
 \Omega 
= \kappa \left( \left\{y\in (0,1)^n \st \sum_{i=1}^n y_i = 1\right\}   \right)
\end{align*}
which we consider as an open subset of $\mathbb{R}^{n-1}$. Similarly, the set of admissible controls $\mathcal{U}$ defined in~\eqref{eq:control_in_ME} is an open subset of~$Z$ after omitting all the components~$u_{i,j}$ that vanish identically. Moreover, condition~{\bf (C1)} is trivially satisfied. Next we recall two 
properties of the time-varying irreducible
master equation (see, e.g.,~\cite{entrainME}). First, Theorem~2.5 in~\cite{entrainME} states in a more general setting the unique existence of a $T$-periodic solution for every admissible control which settles condition~{\bf (C2)}. Second, the combination of Corollary A.6 and Proposition A.7 in~\cite{entrainME} shows that~$\frac{\partial x}{\partial x_0} (T;u,\Gamma(u))=\Phi_{u, \Gamma(u)}(T)$ of condition~{\bf (C3)} does not map any non-zero vector in~$L_H(0)$ to a vector of the same length. Here~$L_H(0)$ is equipped with the norm that is induced by the $\ell_1$-norm of the ambient space~$\mathbb{R}^n$. Thus~$\frac{\partial x}{\partial x_0} (T;u,\Gamma(u))$ cannot have an eigenvalue equal to~$1$. Theorem~\ref{thm:dv} therefore implies
\begin{corollary}[First-order GOE for a constant control in the master equation is zero]\label{Cor:Master}
Consider the master equation
given by~\eqref{eq:MasterEqODEs1new}, \eqref{eq:AdefMaster}, with  the set of~$T$-periodic controls~\eqref{eq:control_in_ME}.
Fix a constant control~$u(t)\equiv \bar v$, and let~$e^{\bar v}$ denote the corresponding equilibrium point. 
For~$\Delta u \in Z$ such that~$u+\Delta u\in \mathcal{U}$,  and~$\overline{\Delta u}:= \frac{1}{T}\int_0^T \Delta u(t)  \diff t=0$,
let~$\gamma^{u+\Delta u}$ denote the corresponding $T$-periodic solution. Then the average outputs along the $T$-periodic 
solutions $y^u(t) := p\left (\bar{v},e^{\bar v}\right )$ and $y^{u+\Delta u}(t) :=p\bigl (u(t)+\Delta u(t),{\gamma}^{u+\Delta u}(t)\bigl)$ satisfy \eqref{eq:GOE1stZero}. 
\end{corollary}
In other words, in the vicinity of
a constant control~$u(t)\equiv \bar{v}$, the GOE in the master equation is a phenomenon of second-order (or higher).

\section{Generalizing the  control system}\label{sec:generalizations}


The form of the control system~\eqref{eq:SystParam} 
was motivated by the~RFM and by Markov chains. Consider now the more general control system:
\begin{align}\label{eq:more_gen_sys}
\dot{x}(t)& =  g(x(t)) +  f(u(t))+ M(x(t)) u(t)  , \\
 y(t)&= h(u(t),x(t)) ,\nonumber
\end{align}
where we added a  state-dependent drift term~$g: \Omega'\rightarrow \mathbb{R}^n$ that is assumed to be~$C^1$. We now show that this can be represented as in~\eqref{eq:SystParam}. Introduce a new control input~$u_0(t)\equiv 1$, and define
\begin{align*}
\hat u: =   \begin{bmatrix}
    u_0\\u
\end{bmatrix}, \ 
\hat M
:= \begin{bmatrix}
g & M    
\end{bmatrix}, \text{ and } 
\Psi := \begin{bmatrix}
    0&1&0&\dots & 0 &0\\
    & & & \vdots & & &\\
    0&0& 0 & \dots & 0  & 1
\end{bmatrix}\in \R^{m\times(m+1)}.
\end{align*}
Then~\eqref{eq:more_gen_sys} can be expressed as 
\begin{align}\label{eq:new_form}
\dot x(t) &= g(x(t)) +  f(u(t))+ M(x(t)) u(t)\nonumber \\
&= f\left(\Psi \hat u(t) \right)+ 
\hat M(x(t)) \hat u(t),
\end{align}
and this is in  the form~\eqref{eq:SystParam} that is analyzed in this paper.

\begin{Example}
Consider again the system in Example~\ref{exa:nonlin_with_goe}, that is, 
    \begin{align*}
    \dot x_1&=-x_1+x_2^2,\nonumber\\
    \dot x_2&=-x_2+u,\nonumber\\
    y&=x_1.
    \end{align*}
This is not in the form~\eqref{eq:SystParam}, but it is in the form~\eqref{eq:more_gen_sys}
with
\[
g(x):=\begin{bmatrix}-x_1+x_2^2 \\-x_2\end{bmatrix}, \; f(u):=\begin{bmatrix} 0\\u \end{bmatrix},\text{ and } 
M(x):=0.
\]
Defining~$u_0(t)\equiv 1$,
\begin{align*}
\hat u: =   \begin{bmatrix}
    u_0\\u
\end{bmatrix}, \ 
\hat M
:= \begin{bmatrix}
 -x_1+x_2^2&0\\
 -x_2 &0
\end{bmatrix}, \text{ and } 
\Psi := \begin{bmatrix}
    0&1 
\end{bmatrix},
\end{align*}
implies that we can write this example in the form~\eqref{eq:new_form}.
In particular, all the theoretical results in Section~\ref{section:main_results} are valid for this example, and this explains why the term for the GOE in~\eqref{eq:quad_goe} is quadratic in the perturbation amplitude~$a$. 
\end{Example}

\section{Conclusion}
Many natural and artificial systems are or can be regulated using periodic controls.
A natural question is: can  periodic controls lead, on average, to a better performance than constant controls? 
Since periodic controls include, as a special case, constant controls, it may seem that the answer to this question is typically yes.

The notion of GOE allows to formulate and analyze this question rigorously. The two key aspects  are: (1)~the controlled system   is assumed to entrain,  so that under a $T$-periodic control all its solutions converge to a unique $T$-periodic solution, and (2)~the comparison between constant and $T$-periodic controls is ``fair'' in the sense that the average value of the control is fixed.

Analysis of the GOE is non-trivial due to several reasons. First, the $T$-periodic solution of the system is usually not known explicitly. Second, the analysis requires to compute the  derivative of the state with respect to a small perturbation in
the control, and this
implies that a general treatment of this  problem is intrinsically  infinite-dimensional.

Here, we studied the GOE in a    class of   systems affine in the control. The main assumption is that the controlled system is contractive. This implies entrainment to~$T$-periodic controls. 
In fact, the  analysis shows that entrainment is almost a sufficient condition to derive our results. It only needs to be supplemented by some condition on the sensitivity function of the Poincar\'e map that is satisfied, e.g.~by some weak form of contractivity as it is present in the master equation for irreducible Markov processes.

We showed that given a constant control~$u(t)\equiv \bar v$
and a $T$-periodic perturbation~$\Delta u$ (with average~$T^{-1}\int_0^T \Delta u(t)\diff t=0$), the first-order term of the~GOE is zero.
The proof is based on  analysis of the associated Poincar\'e map.

Our results suggest that 
certain systems may have a GOE that is \emph{always}  positive [negative]   corresponding to certain ``convexity'' [``concavity''] like  properties of the second derivative of the state w.r.t.  perturbations of the control. In such systems  periodic controls will \emph{always} be better [worse] than constant controls, on average. 
Determining the structure of such systems is an interesting and non-trivial research problem. 

We demonstrated our results using the master equation and  a  phenomenological model for 1D transportation called the ribosome flow model~(RFM), both with~$T$-periodic rates. 
We believe that the results hold for other  examples of contractive  systems as well. 

\subsection*{Acknowledgements}  We thank Alexander Ovseevich   for helpful comments.

 \onehalfspacing


\end{document}